\documentclass[AMA,Times1COL]{WileyNJDv5}  %
\usepackage{amsmath,amsfonts}
\usepackage{array}
\usepackage[FIGTOPCAP]{subfigure}
\usepackage{textcomp}
\usepackage{url}
\usepackage{verbatim}
\usepackage{graphicx}
\usepackage{algorithm}
\usepackage{algpseudocode}
\usepackage{multirow}

\usepackage[short]{optidef}
\usepackage{multicol}
\usepackage{comment}
\usepackage[]{xcolor}
\usepackage{caption}
\usepackage{wrapfig}
\usepackage{booktabs}

\definecolor{Burgundy}{RGB}{144,0,32}

\usepackage{acronym}

\newacro{DM}{decision making}
\newacro{AD}{automated driving}
\newacro{PDG}{powered descent guidance}
\newacro{ODE}{ordinary differential equation}
\newacro{RK4}{Runge-Kutta of order 4 integrator}
\newacro{SOS1}{special-ordered-sets of type 1}
\newacro{S-B-MIQP}{sequential Benders-based mixed-integer quadratic programming}
\newacro{NBB}{nonlinear branch-and-bound}
\newacro{GBD}{generalized Benders decomposition}
\newacro{KKT}{Karush-Kuhn-Tucker}
\newacro{DoF}{degrees of freedom}

\newacro{NLP}{nonlinear program}
\newacro{MPVC}{mathematical program with vanishing constraints}
\newacro{MPCC}{mathematical program with complementarity constraints}
\newacro{MPC}{nonlinear model predictive control}
\newacro{NMPC}{nonlinear model predictive control}
\newacro{MPPI}{model predictive path integral control}
\newacro{QP}{quadratic program}
\newacro{LP}{linear program}
\newacro{GDP}{generalized disjunctive programming}
\newacro{MIQP}{mixed-integer quadratic program}
\newacro{MILP}{mixed-integer linear program}
\newacro{MINLP}{mixed-integer nonlinear program}
\newacro{MI}{mixed-integer}
\newacro{MIP}{mixed-integer programming}
\newacro{BB}{branch-and-bound}
\newacro{SQP}{sequential quadratic programming}
\newacro{RNN}{recurrent neural network}
\newacro{OCP}{optimal control problem}
\newacro{LQR}{linear quadratic regulator}
\newacro{iLQR}{iterative linear quadratic regulator}
\newacro{SLQ}{sequential linear quadratic programming}
\newacro{DDP}{differential dynamic programming}
\newacro{MCP}{mixed complementarity problem}
\newacro{IP}{interior point}
\newacro{ADMM}{alternating direction method of multipliers }
\newacro{RTI}{real time iteration}
\newacro{MDP}{Markov decision process}
\newacro{SOC}{second order cone}
\newacro{SOCP}{second order cone problem}

\newacro{RL}{reinforcement learning}
\newacro{DP}{dynamic programming}
\newacro{LSTM}{long short-term memory}
\newacro{NN}{neural network}
\newacro{PPO}{proximal policy optimization}
\newacro{PG}{policy gradient}
\newacro{TD}{temporal difference}
\newacro{SAC}{soft actor critic}
\newacro{PI}{policy iteration}
\newacro{SARSA}{state action reward state action}
\newacro{IL}{imitation learning}
\newacro{MLE}{maximum likelihood}
\newacro{SV}{surrounding vehicle}
\newacro{EV}{ego vehicle}

\newacro{FF}{feed forward network}
\newacro{DNN}{deep neural network}

\newacro{UGV}{unmanned ground vehicle}
\newacro{UAV}{unmanned aerial vehicle}

\newacro{AC4MPC}{Actor Critic for Nonlinear Model Predictive Control}

\newcommand{\R}{\mathbb{R}}
\newcommand{\Z}{\mathbb{Z}}

\newcommand{\norm}[1]{\left\lVert #1 \right\rVert}
\newcommand{\indfn}{\sigma}
\newcommand{\approxindfn}{\tilde{\sigma}}
\DeclareMathSymbol{\shortminus}{\mathbin}{AMSa}{"39}

\articletype{Article Type}%
\received{00}
\revised{00}
\accepted{00}

\begin{document}
\title{A Comparative Study of MINLP and MPVC Formulations for Solving Complex Nonlinear Decision-Making Problems in Aerospace Applications}

\author[1]{Andrea Ghezzi}
\author[1]{Armin Nurkanovi\'c}
\author[2]{Avishai Weiss}
\author[3]{Moritz Diehl}
\author[2]{Stefano Di Cairano}

\address[1]{\orgdiv{Department of Microsystems Engineering (IMTEK)}, \orgname{University of Freiburg}, \orgaddress{\country{Germany}}}
\address[2]{\orgdiv{Mitsubishi Electric Research Laboratories (MERL)}, \orgname{Cambridge}, \orgaddress{\state{MA}, \country{USA}}}
\address[3]{\orgdiv{Department of Microsystems Engineering (IMTEK) and Department of Mathematics}, \orgname{University of Freiburg}, \orgaddress{\country{Germany}}}
\corres{andrea.ghezzi@imtek.uni-freiburg.de}

\abstract[Abstract]{
    High-level decision-making for dynamical systems often involves performance and safety specifications that are activated or deactivated depending on conditions related to the system state and commands.
    Such decision-making problems can be naturally formulated as optimization problems where these conditional activations are regulated by discrete variables.
    However, solving these problems can be challenging numerically, even on powerful computing platforms, especially when the dynamics are nonlinear.
    In this work, we consider decision-making for nonlinear systems where certain constraints, as well as possible terms in the cost function, are activated or deactivated depending on the system state and commands.
    We show that these problems can be formulated either as mixed-integer nonlinear programs (MINLPs) or as mathematical programs with vanishing constraints (MPVCs), where the former formulation involves discrete decision variables, whereas the latter relies on continuous variables subject to structured nonconvex constraints.
    We discuss the different solution methods available for both formulations and demonstrate them on optimal trajectory planning problems in various aerospace applications.
    Finally, we compare the strengths and weaknesses of the MINLP and MPVC approaches through a focused case study on powered descent guidance with divert-feasible regions.
    In our simulations for problems up to medium size, MPVC formulations provide accurate solutions faster than MINLP formulations.
    However, for larger problems, the MPVC formulation introduces numerous nonconvexities that hinder solver convergence, even when they are relatively simple, making MINLPs the preferred choice in such cases.
    }

\maketitle

\section{Introduction}
We consider constraint-triggered optimization problems described using logical expressions
\begin{mini!}[2]
    {z \in \mathbb{R}^{n_z}}{f(z) - \sum_{i=1}^{n_\delta} w_i \indfn(H_i(z))}{\label{op: intro}}{\label{cf: extended nlp with logic implication}}
    \addConstraint{g(z) }{= 0, \; h(z) \leq 0\label{cns: intro normal}}
    \addConstraint{H_i(z)}{\geq 0 \Rightarrow G_i(z) \leq 0,}{\quad i \in \Z_{[1, n_\delta]},\label{cns: intro implication}}
\end{mini!}
where $\indfn$ is the Heaviside step-function $\indfn: \R \to \{0, 1\}$,
\begin{align}\label{eq: heaviside step function}
    \indfn(y) \coloneqq
    \begin{cases}
        1,  \quad &\text{if } y \geq 0, \\
        0, \quad &\text{if } y < 0.
    \end{cases}
\end{align}
In~\eqref{op: intro}, $w_i \geq 0$ are weight coefficients for the second term in the objective function, and functions $f: \R^{n_z} \to \R$, $g: \R^{n_z} \to \R^{n_g}$, $h: \R^{n_z} \to \R^{n_h}$, $H_i: \R^{n_z} \to \R$, $G_i: \R^{n_z} \to \R^{n_{G_i}}$ are assumed to be twice continuously differentiable.
Based on the chosen representation of the Heaviside step-function in the cost~\eqref{cf: extended nlp with logic implication}, and of the logical expression in the constraint~\eqref{cns: intro implication}, the mathematical program~\eqref{op: intro} can be formulated as a \ac{MINLP} or a \ac{MPVC}.
The convention $\sigma(0)=1$ makes~\eqref{eq: heaviside step function} an upper semi-continuous function, thus the minimization problem~\eqref{op: intro} is well-defined.
In some applications (cf. Sec.~\ref{subsec: problem descriptions}), we consider the left-hand-side of constraint~\eqref{cns: intro implication} to hold with strict inequality ``$H_i(z) > 0$''.
Strict inequalities define an open feasible set which might cause some issues in view of the existence of solutions of problem~\eqref{op: intro}.
However, the \ac{MINLP} and \ac{MPVC} formulations of~\eqref{op: intro} presented in Sec.~\ref{sec: formulations} will lead to well-defined problems.
The distinctive feature of~\eqref{op: intro} lies in the maximization of the Heaviside step-function within the cost.
By tuning the coefficients $w_i$, one trade-offs between finding optimal solutions to the \ac{NLP} with cost function $f$ and constraints~\eqref{cns: intro normal}, and the enforcement of additional indicator constraints expressed by functions $H_i$, that in turn imply constraints expressed by functions $G_i$.

A special case yet important case of~\eqref{op: intro} is when indicator constraints $H_i(z) \geq 0$ corresponds to indicator variables.
Consider the optimization variable $\tilde{z} = (z, \delta) \in \R^{n_z + n_\delta}$ such that $H_i(\tilde{z}) = \delta_i$ and $\delta_i \in [0, 1]$, for all $i \in \Z_{[1, n_\delta]}$, thus~\eqref{op: intro} can be written as
\begin{mini!}[2]
    {(z, \delta) \in \R^{n_z + n_\delta}}{f(z) - \sum_{i=1}^{n_\delta} w_i \delta_i}{\label{op: intro special case}}{\label{cf: extended nlp with logic implication special case}}
    \addConstraint{g(z) }{= 0, \; h(z) \leq 0\label{cns: intro normal special case}}
    \addConstraint{\delta_i}{> 0 \Rightarrow G_i(z) \leq 0,}{\quad i \in \Z_{[1, n_\delta]},\label{cns: intro implication special case}}
    \addConstraint{0 \leq \delta_i}{\leq 1,}{\quad i \in \Z_{[1, n_\delta]},\label{cns: intro bounds on delta special case}}
\end{mini!}
Problem~\eqref{op: intro special case} does not contain the Heaviside step-function from~\eqref{cf: extended nlp with logic implication special case}, making~\eqref{op: intro special case} easier to treat than~\eqref{op: intro}.
A similar remark to one above can be made for the strict inequality in \eqref{cns: intro implication special case}.
Additionally, we will show that optimal solution of both MINLP and MPVC reformulation of~\eqref{op: intro special case} are characterized by $\delta_i \in \{0, 1\},$ for all $i \in \Z_{[1, n_\delta]}$.

In this work, we introduce optimal trajectory planning problems for aerospace applications that can be modeled by~\eqref{op: intro} and~\eqref{op: intro special case}, and demonstrate how the resulting optimization problems are solved.

\textbf{Related work --}
Mathematical programs with indicator constraints, or ``on/off'' constraints, such as~\eqref{cns: intro implication} have been thoroughly studied in the operation research literature~\cite{Hijazi2012, Bonami2015, Belotti2016}, and it is well-known that logical implications can be represented either via disjunctive programming~\cite{Balas1979,Grossmann2002} or complementarity and vanishing constraints~\cite{Scholtes2004,Achtziger2008}.
Among the disjunctive programming approaches, the most common and easiest to implement is the \textit{big-M} method~\cite{Wolsey1999}.
However, the \textit{big-M} method leads to weak relaxations~\cite{Bonami2015}.
Other approaches for representing disjunctions are based on perspective formulations~\cite{Ceria1999,Grossmann2013} and lead to tailored algorithms such as branch-and-cut~\cite{Stubbs1999, Frangioni2006}.
Nevertheless, approaches based on the perspective formulation require constructing the convex hull of the disjunction, which can be computed analytically only for specific sets~\cite{Gunluk2011, Hijazi2012, Bonami2015}.
If the analytical expression is not available, the convex hull might be constructed iteratively via cutting planes as in~\cite{Stubbs1999, Frangioni2006}, thus the resulting cutting planes have to be integrated in a branch-and-cut algorithm.
The iterative procedure to obtain cutting planes for the convex hull might require similar computation time to that of the big-M formulation.
The approach of using complementarity constraints for representing logical implications requires the introduction of structured nonconvexities that violate standard constraint qualifications, such that even robust \ac{NLP} solvers might fail to converge to a solution~\cite{Fletcher2006, Hoheisel2008},~\cite[\S 9.3]{Kirches2011}.
In practice, the use of complementarity constraints is advised for problems that possess a connected feasible set.
Despite the numerical issues related to the solution of problems with complementarity constraints, a major advantage is that they require the solution of \acp{NLP} only, thus there is no reliance on \ac{MIP} solvers.

In the context of this paper,~\eqref{op: intro} arises from the time-discretization of an \ac{OCP} via a direct method, e.g., direct multiple shooting~\cite{Bock1984} or direct collocation~\cite{Tsang1975}.
This type of optimization problem aims at finding the control inputs to a dynamic system that produce the best possible performance over a specified time horizon, subject to constraints on the system state and controls.
Similar to this work, in~\cite{Jung2013} the authors present a \ac{MINLP} and a \ac{MPVC} formulation of a mixed-integer \ac{OCP} whose discrete part is not limited to the representation of logical expressions but also to system states and controls.
The focus of~\cite{Jung2013} relies in the comparison of the relaxations obtained from the two alternative formulations.
In this work, we focus on \acp{OCP} that are continuous except for the indicator variables, and we aim at comparing computational methods and quality of the final solutions for both \ac{MINLP} and \ac{MPVC} formulations.
In Section~\ref{sec: trajectory planning problems}, we specify~\eqref{op: intro} as an \ac{OCP} to provide a reference formulation.

The main case study in this paper is the powered descent phase in Mars landing~\cite[p. 87]{Malyuta2022}.
In \ac{PDG}, optimal control and numerical optimization have been largely adopted for computing time- and fuel-optimal trajectories~\cite{Acikmese2007}.
The method proposed in~\cite{Acikmese2007} formulates \ac{PDG} as a convex problem that can be solved in milliseconds to global optimality~\cite{Acikmese2008}, and the formulation has been further extended to handle non-convex control constraints~\cite{Acikmese2011}, linear and quadratic state constraints~\cite{Harris2014, Harris2013}.
For more references on \ac{PDG} we refer the reader to the tutorial article~\cite{Malyuta2022}.
The convexification adopts a 3 \ac{DoF} point mass model, which is generally considered accurate enough for this application.
Despite the numerous enhancements to the convex reformulation, at the current status it is not possible to obtain a ``lossless convexification'' in the case of continuous activation of constraints~\cite{Malyuta2022}, i.e., if an optimal solution has active path constraints in consecutive time steps.
Also, if the objective of the problem is modified, such that it no longer exclusively seeks fuel- and time-optimal trajectories, a new tailored convexification must be developed which is not guaranteed to exist.
In this work, we address a \ac{PDG} problem augmented with divert-feasible regions.
The spacecraft must approach the landing site while traversing the maximum number of these regions, areas from which alternative landing sites are reachable, thereby maximizing the ``divert possibility''.
Such problem can be formulated as~\eqref{op: intro} which, to the authors' knowledge, cannot be losslessly convexified, and thus must be addressed as is, either as a \ac{MINLP} or a \ac{MPVC}.

\textbf{Contribution --}
The main contribution of this paper is a discussion on how to solve complex nonlinear decision-making problems, within acceptable computation times, and using existing software packages.
We present relevant case studies in the aerospace domain that can be modeled as~\eqref{op: intro}, and we draw a continuous parallel between their formulation as \acp{MINLP} and \acp{MPVC}, highlighting the differences and peculiarities of each approach.
The more comprehensive case study is the \ac{PDG} with divert-feasible regions for Mars landing, a problem that is currently the subject of active research due to the renewed interest in Mars exploration, and for which we propose a formulation enabled by nonlinear modeling.
Through extensive numerical simulations, we demonstrate the efficacy of these formulations and provide insights on when a \ac{MINLP} or a \ac{MPVC} formulation is more suitable.

\textbf{Outline --}
In Section~\ref{sec: trajectory planning problems}, we specialize formulation~\eqref{op: intro} for \acp{OCP}, and introduce relevant aerospace case studies that can be modeled as~\eqref{op: intro} and~\eqref{op: intro special case}.
Section~\ref{sec: formulations} shows how to obtain a numerically tractable expression for the logical implication~\eqref{cns: intro implication} and for the Heaviside-step function \eqref{eq: heaviside step function}.
Moreover, we present two alternative formulations of both~\eqref{op: intro} and~\eqref{op: intro special case}: one resulting in a \ac{MINLP} and one in a \ac{MPVC}.
In Section~\ref{sec: solution methods}, we survey some solution methods for \acp{MINLP} and \acp{MPVC}, and compare them for a tutorial example.
Section~\ref{sec: landing results} describes a case study of the \ac{PDG} problem constrained with divert-feasible regions, obtains both the \ac{MINLP} and \ac{MPVC} formulations, and solves them with the presented solution methods.
Finally, Section~\ref{sec: conclusions} contains concluding remarks and possible future research.

\textbf{Notation --}
The set set of positive real numbers is denoted as $\R_+$.
With ~$\Z_{[a, b]}$, $a < b$, we denote the interval of integers $\{a, a+1, \dots, b\}$.
Vector inequalities are intended component-wise.
The symbol $\Rightarrow$ denotes the logical implication, $\indfn$ the Heaviside step-function, and $\norm{\cdot}_2$ the vector Euclidean norm.
Given two variables $a, b$, we write a complementarity condition using the symbol $\perp$ as $0 \leq a \perp b \geq 0$, i.e., it must hold either $a \geq 0, b=0$, or $a=0, b\geq 0$.

\section{Decision-making via optimal control}\label{sec: trajectory planning problems}
In this section, we present trajectory planning case studies relevant in the aerospace domain that can be formulated as~\eqref{op: intro}.
Each trajectory planning task is formulated as an \acf{OCP} that is solved via a direct method.
First, we introduce how to specialize the generic formulation~\eqref{op: intro} to an \ac{OCP} formulation.
Then, we present the trajectory planning problems.

\subsection{Discrete-time optimal control problem}\label{subsec: discrete-time OCP}
We consider dynamical systems modeled via \acp{ODE}
\begin{equation*}
    \dot{x}(t) = f(x(t), u(t)),
\end{equation*}
where $t$ denotes the time, $x \in \R^{n_x}$ and $u \in \R^{n_u}$ the system state and control vectors, respectively.
To optimally control the system, we formulate an \ac{OCP} over a time interval $t \in [0, t_\mathrm{f}]$ that encodes performance goals via the cost function, and safety specifications via constraints.
In order to numerically solve such \ac{OCP}, we adopt direct multiple shooting~\cite{Bock1984}.
By means of a suitable integration method, e.g., a Runge-Kutta integrator, we discretize the \acp{ODE}, the cost function, and the constraints over a uniform grid with $N$ intervals where $0 = t_0 < t_1 < ... < t_N = t_\mathrm{f}$.
The discretization step is denoted by $t_\mathrm{d}$ and is defined as $t_\mathrm{d} \coloneqq t_\mathrm{f} / N$.
The final time $t_\mathrm{f}$ may be included as optimization variable to formulate time-optimal problems.
After the discretization step, we obtain a mathematical program as~\eqref{op: intro} but with a specific \ac{OCP} structure,
\begin{mini!}[2]
    {t_\mathrm{f}, \mathbf{x},\mathbf{u}}{E(x_N, t_\mathrm{f}) + \sum_{k=0}^{N-1} \left(L(x_k, u_k, \tfrac{t_\mathrm{f}}{N}) - \frac{t_\mathrm{f}}{N}\sum_{i=1}^{n_\delta} w_{i, k} \indfn(r_i(x_k, u_k)) \right)}{\label{op: OCP-NLP template}}{\label{cfn: cost function OCP-NLP template}}
    \addConstraint{x_0}{= \bar{x}_0}
    \addConstraint{x_{k+1}}{= F(x_k, u_k, \tfrac{t_\mathrm{f}}{N}),}{\quad k \in \Z_{[0, N-1]}}
    \addConstraint{c(x_k, u_k)}{\leq 0,}{\quad k \in \Z_{[0, N-1]}}
    \addConstraint{r_i(x_k, u_k) \geq 0}{\Rightarrow s_i(x_k, u_k) \leq 0,}{\quad i \in \Z_{[1, n_\delta]}, k \in \Z_{[0, N-1]}\label{cns: indicator variable ocp-template}}
    \addConstraint{c_N(x_N)}{\leq 0,}
\end{mini!}
where $\mathbf{x} \coloneqq (x_0, \dots, x_N)$ and $\mathbf{u} \coloneqq (u_0, \dots, u_{N-1})$ are the stage-wise concatenation of the state and control variables, respectively.
For the simplicity of notation, we do not consider implication constraints involving the terminal state, but the extension to handle this case is straightforward.
The cost function~\eqref{cfn: cost function OCP-NLP template} includes a terminal cost $E: \R^{n_x} \times \R \to \R$ and a stage cost with two distinct terms.
The first term $L: \R^{n_x} \times \R^{n_u} \times \R \to \R$ depends only on state, control, and discretization step.
The second term is a weighted sum of the Heaviside step-functions.
Function $F: \R^{n_x} \times \R^{n_u} \to \R^{n_x}$ is obtained from a suitable integration method that discretizes the corresponding \ac{ODE}, e.g., a Runge-Kutta integrator.
Function $c: \R^{n_x} \times \R^{n_u} \to \R^{n_c}$ defines nonlinear path constraints, enforced at grid nodes, and $c_N: \R^{n_x} \to \R^{n_{c_N}}$ expresses a terminal constraint.
Functions $s_i: \R^{n_x} \times \R^{n_u} \to \R^{n_s}$ construct constraints that are enforced only when the corresponding triggering function $r_i: \R^{n_x} \times \R^{n_u} \to \R$ is satisfied.
Similarly to~\eqref{op: intro special case}, it is straightforward to obtain a special case for~\eqref{op: OCP-NLP template} when the indicator constraints $r_i(x_k, u_k) \geq  0$ corresponds to indicator variables $\delta_i \in [0, 1]$.
We consider formulation~\eqref{op: OCP-NLP template} as a template for the problems presented in the following subsections.

\subsection{Aerospace trajectory planning case studies}\label{subsec: problem descriptions}
Next, we describe four trajectory planning problems.
The first two correspond to the special case of~\eqref{op: OCP-NLP template}, as they lack indicator constraints defined by functions $r_i$ and instead use indicator variables $\delta_i \in [0, 1]$.

\subsubsection{Powered descent guidance with divert-feasible trajectories}
We aim at computing landing trajectories for a spacecraft that in addition to minimizing flight time and fuel usage, approach the primary landing site by traversing for as much time as possible divert-feasible regions, i.e., regions from which the spacecraft can divert to alternative landing sites.
These divert-feasible regions may be represented by polytopes obtained by reachability analysis~\cite{Srinivas2024}.
We define $p_k \in \R^3, v_k \in \R^3$, where $k\in \Z_{[0, N-1]}$, as the 3D position and velocity of the spacecraft in an inertial reference frame with origin in the primary landing site.
Additionally, we define $\xi \coloneqq (p, v)$ as the concatenation of position and velocity.
At each point of the time grid, $t_k$, with $k \in \Z_{[0, N-1]}$,  we consider $n_\mathrm{p}$ polytopes in the half-space representation: $A_{k, i} \xi_k + b_{k, i} \leq 0, \, i\in \Z_{[1, n_\mathrm{p}]}$, where $A_{k, i} \in \R^{n_i \times 6}$, $b_{k, i} \in \R^{n_i}$, and $n_i$ is the number of halfspaces in the $i$-th polytope.
We introduce one indicator variable $\delta_{k, i} \in [0, 1]$ to model the membership of the state $\xi_k$ to the $i$-th polytope at the $k$-th time instant,
\begin{equation}
    \delta_{k, i} > 0\Rightarrow A_{k, i} \xi_k + b_{k, i} \leq 0, \quad k\in \Z_{[0, N-1]}, i\in \Z_{[1, n_\mathrm{p}]}.
\end{equation}

Finally, to obtain trajectories that traverse for the longest possible time multiple divert-feasible regions, we specify the second term of the stage cost in~\eqref{cfn: cost function OCP-NLP template} as
\begin{equation}\label{eq: stage cost to max deltas}
    \frac{t_\mathrm{f}}{N}\sum_{i=1}^{n_\mathrm{p}} w_{k, i}\delta_{k, i}, \quad k \in \Z_{[0, N-1]},
\end{equation}
where $w_{k, i} \geq 0$ are weighting coefficients.
In Section~\ref{sec: landing results}, we describe in detail the \ac{PDG} constrained with divert-feasible regions and present numerical simulations to compare the MINLP approach against the MPVC approach.

\subsubsection{Coordination of unmanned ground and aerial vehicles}
We consider the coordination problem between an \ac{UGV} and several \acp{UAV} described in~\cite{Kim2024}.
A set of monitoring targets must be visited by \acp{UAV}.
These are carried by a UGV which serves as a mobile docking site for recharging.
We aim to plan a trajectory for the UGV such that the \acp{UAV} can visit all the targets.

For each monitoring target it is possible to construct a reachable set for the \acp{UAV}.
The set expresses the region of space where the \ac{UAV} can leave the UGV, accomplish the monitoring task, and land again on the UGV with a prescribed minimum state of charge of the battery.
We consider inner convexifications of the reachable sets.
These convexifications may be more conservative but are easier to handle in optimization problems.
Therefore, it is possible to plan the trajectory of the UGV from a start to an end point that encompasses each reachable set at least once.
Also, the trajectory has to maximize the time spent by the UGV in the reachable sets, and further constraints can be imposed, e.g., actuation limits, obstacle avoidance.

We denote the position of the UGV by $p_k \in \R^2, k\in \Z_{[0, N-1]}$, and we consider $n_\mathrm{p}$ reachable sets, one for each target, described by polytopes as $A_i p_k + b_i \leq 0, i = \Z_{[1, n_\mathrm{p}]}$.
Differently from the landing problem, we assume the reachable set to be independent from time.
We model the membership of the UGV position in the reachable sets of the targets by introducing an indicator variable $\delta_{k, i} \in [0, 1]$,
\begin{equation}\label{eq: ugv logic cns generic 1}
    \delta_{k, i} > 0\Rightarrow A_{i} p_k + b_{i} \leq 0, \quad k\in \Z_{[0, N-1]}, i \in \Z_{[1, n_\mathrm{p}]}.
\end{equation}
To ensure that each reachable set is visited at least once along the UGV's trajectory we enforce
\begin{equation}\label{eq: ugv logic cns generic 2}
    \sum_{k=0}^{N-1} \delta_{k, i} \geq 1, \quad i \in \Z_{[1, n_\mathrm{p}]}.
\end{equation}
To maximize the time spent by the UGV in the reachable sets, the second term of the stage cost in~\eqref{op: OCP-NLP template} is identical to~\eqref{eq: stage cost to max deltas}.
Through functions $c$ and $c_N$ we can impose additional constraints, and through functions $L$ and $E$ additional stage and terminal objectives, respectively.

\subsubsection{Soft docking}
We consider a spacecraft docking where the goal is to compute a fuel- and time-optimal trajectory that successfully docks a chaser spacecraft to a target spacecraft.
Specifically, we consider the formulation presented in~\cite{Malyuta2020} where we substitute the so-called silent thruster constraint with a general soft-docking constraint~\cite{DiCairano2012}.
The latter ensures a reduction of the maximum velocity of the chaser spacecraft as the target spacecraft is approached.
The other important constraint in spacecraft docking is the line-of-sight constraint, which enforces a specific range for the approach angle of the spacecraft with respect to the docking port to ensure proper sensing.
To avoid over-conservative trajectories the two docking-specific constraints, i.e., soft-docking and line-of-sight, are imposed only in the proximity of the docking site.

We denote by $p_k \in \R^3, v_k \in \R^3, k\in \Z_{[0, N-1]}$ the 3D position and velocity of the spacecraft in a global inertial reference frame.
The soft-docking constraint is formulated as
\begin{equation}\label{cns: soft docking}
    \norm{p_k - p_\mathrm{f}}_2 \leq r \Rightarrow \norm{v_k}_2 \leq \alpha \norm{p_k - p_\mathrm{f}}_2, \quad k\in \Z_{[0, N-1]},
\end{equation}
where $p_\mathrm{f}$ is the position of the docking site, $r \in \R_+$ is a prescribed distance from the docking site, and $\alpha \in \R_+$ is a parameter for the relationship between the maximum velocity and the distance from the docking site.
Similarly the line-of-sight constraint is
\begin{equation}\label{cns: plume impingement}
    \norm{p_k - p_\mathrm{f}}_2 \leq r \Rightarrow \norm{p_k - p_\mathrm{f}}\cos(\theta_\mathrm{max}) \leq (p_k - p_\mathrm{f})^\top e_\mathrm{f}, \quad k\in \Z_{[0, N-1]},
\end{equation}
where $\theta_\mathrm{max}\in(0, \tfrac{\pi}{2})$ is the approach-cone half angle and $e_\mathrm{f} \in \R^3$ is the docking site axis orientation.

For this application the cost term involving indicator variables in~\eqref{cfn: cost function OCP-NLP template} is omitted, since these are not performance objectives but rather safety specifications.

In the aerospace engineering literature, constraints~\eqref{cns: soft docking},~\eqref{cns: plume impingement} are often called ``state-triggered constraints''~\cite{Szmuk2020}, as the constraints are only enforced if the current system satisfies a specific ``trigger'' condition which depends on the system state.
However, in~\cite{Szmuk2020} the constraints on the right side of the logical implications hold with equality.
We remark that this problem could also be modelled as a multi-phase \ac{OCP} by choosing a priori when the spacecraft has to be in the proximity of the docking site.
The transition can be imposed via equality constraints on the system state at a specific time instant.
Hence, we obtain an \ac{OCP} without logical expressions which results in a standard \ac{NLP} after time discretization.
But, the multi-phase \ac{OCP} has reduced flexibility, and thus, in general, lower performance.

\subsubsection{Abort-safe spacecraft rendezvous}
By following NASA's convention for a rendezvous of a spacecraft (``deputy'') with the International Space Station (``chief''), we consider two rendezvous phases~\cite{Fehse2003}.
In the first phase the deputy has to be \textit{passively} safe, meaning that in case of a complete loss of all thrusters, the deputy will not collide with the chief.
In the second phase, once the deputy is close enough to the chief, often it is not possible to ensure the existence of \textit{passively} safe trajectories.
Thus, we aim at computing deputy trajectories that are \textit{actively} safe, i.e., even in case of partial thrust failure they can avoid collision with the chief by actuating the remaining thrusters.
The maximum number of thrust failures allowed is fixed when the second phase starts.
To determine which regions of the space are abort-safe it is possible to construct backward reachable sets~\cite{Aguilar2022}.
The sets are constructed in different ways for both passive and active safety, but here we assume for simplicity that they are given and represented by polytopes.
As usual, the planned trajectory of the deputy, in addition to being abort-safe, should also be fuel- and time-optimal.

We model the trajectory planning for a safe rendezvous as follows.
First, we introduce a logical implication to determine whether the deputy should be passively or actively safe based on the distance from the chief.
For passive safety, the deputy trajectory in each time step must lie in one of the passive reachable sets.
Conversely, for active safety, the deputy trajectory must lie in the active reachable set corresponding to the prescribed level of acceptable thrust failure.
In addition, it might be possible to enhance robustness against failures by promoting trajectories that, when possible, lie also in other active reachable sets corresponding to higher levels of thrust failures.

Mathematically, we state the above as follows.
First, we define by $\xi_k \coloneqq (p_k, v_k), k \in \Z_{[0, N-1]}$ the concatenation of position and velocity of the deputy, respectively, and by $\mathcal{P}_k, i \in \Z_{[1, n_\mathrm{p}]}$ the reachable sets corresponding to passively safe regions.
For detecting and enforcing passive safety we impose
\begin{equation}\label{cns: passive safety rendezvous}
    \norm{p_k - p_\mathrm{c}}_2 > r \Rightarrow \xi_k \in \mathcal{P}_k, \quad k\in \Z_{[0, N-1]},
\end{equation}
where $p_\mathrm{c} \in \R^3$ is the position of chief, and $r$ is the distance that divides the region requiring passive safety from the one requiring active safety.
Conversely, to detect and enforce active safety we impose
\begin{equation}\label{cns: active safety rendezvous}
    \norm{p_k - p_\mathrm{c}}_2 \leq r \Rightarrow \xi_k \in \mathcal{A}_k, \quad k\in \Z_{[0, N-1]},
\end{equation}
where $\mathcal{A}_k$ corresponds to the active-safe set that must be satisfied.
Often, $\mathcal{A}_k$ is formed by multiple regions as $\mathcal{A}_k = \bigcup_{i=1}^{n_\mathrm{a}} \hat{\mathcal{A}}_{k, i}$, hence~\eqref{cns: active safety rendezvous} can be further specified as
\begin{equation}
    \norm{p_k - p_\mathrm{c}}_2 \leq r \Rightarrow \xi_k \in \hat{\mathcal{A}}_k, \delta_{k, i} \geq 0, \quad k\in \Z_{[0, N-1]}, i \in \Z_{[1, n_\mathrm{a}]},
\end{equation}
where $\delta_{k, i} \in [0, 1]$ are indicator variables that denote the membership in the active-safe sets $\hat{\mathcal{A}}_{k, i}, i \in \Z_{[1, n_\mathrm{a}]}$.
Specifically, the indicator variables enforce the constraints
\begin{equation}
    \delta_{k, i} > 0\Rightarrow \xi_k \in \hat{\mathcal{A}}_{k, i}, \quad k \in \Z_{[0, N-1]}, i \in\Z_{[1, n_\mathrm{a}]},
\end{equation}
and we guarantee active safety by requiring
\begin{equation}
    \sum_{i=1}^{n_\mathrm{a}} \delta_{i, k} \geq 1, \quad k \in \Z_{[0, N-1]}.
\end{equation}
Trajectories that traverse the intersection of multiple active sets $\hat{\mathcal{A}}_{k, i}$ are considered safer as intersections may represent reachable sets for higher levels of thrust failures.
In order to encourage the optimizer to find such trajectories, we express the second term of the objective in~\eqref{op: intro} as
\begin{equation}
    \frac{t_\mathrm{f}}{N}\sum_{i=1}^{n_\mathrm{a}} w_{k, i} \delta_{k, i}, \quad k \in \Z_{[0, N-1]},
\end{equation}
where $w_{k, i} \geq 0$ are weighting coefficients.
Thus, this problem combines logical constraints regulated by indicator variables and indicator constraints.

\section{Obtaining computationally tractable logical expressions}\label{sec: formulations}
In this section, we demonstrate how to translate the logical implications contained in~\eqref{op: intro} into an optimization problem suitable for numerical solvers.
We start by presenting formulations of~\eqref{op: intro special case} where the constraints involving the logical implications~\eqref{cns: intro implication} are governed solely by indicator variables.
The general formulation~\eqref{op: intro} is addressed at the end of this section since the presence of indicator constraints and the Heaviside step-function in the cost require the introduction of approximations.

There are numerous strategies to handle logical expressions in the literature, here we consider two.
The first one relies on the \ac{GDP} framework introduced by Balas~\cite{Balas1979} for \acp{MILP}, and extended by Grossmann and coworkers to \acp{MINLP}~\cite{Grossmann2002, Grossmann2012}.
In \ac{GDP} logical expressions can be stated via the so-called ``big-M'' constraints or via perspective functions based on the convex hull formulation.
In this work, we employ big-M constraints since their use is straightforward and avoids the introduction of copies of variables, differently from convex hull formulations~\cite{Vielma2015}.
This is particularly important as we consider problems with a large number of decision variables and aim to solve them quickly.
A potential drawback of big-M formulations is the assumption that the problem's constraints have an upper (or lower) bound.
However, this issue is generally not significant for control problems, where variables and constraints are naturally bounded by physical limits.
A second disadvantage of big-M formulations is the low quality of their relaxations, which often result in excessively optimistic lower bounds.
Loose relaxations can directly increase solver time.
However, \acp{OCP} typically admit tight bounds on states and controls, which also lead to tight bounds on other constraints, making big-M formulations more attractive.
Furthermore, when state-of-the-art \ac{MIP} solvers are used, the presolve routines can preprocess the given formulation yielding tighter relaxations~\cite{Achterberg2020}.

Let functions $G_i$ in~\eqref{cns: intro implication special case} admit an upper bound $M \in \R_+$ such that, for every admissible value $z$, $G_i(z) \leq M$, $i\in \Z_{[1, n_\delta]}$.
Then,~\eqref{cns: intro implication special case} is reformulated as
\begin{equation}
    G_i(z) \leq M (1 - \delta_i), \quad i\in \Z_{[1, n_\delta]}.
\end{equation}
When $\delta_i = 1$, the constraint $G_i(z) \leq 0$ is enforced.
When $\delta_i = 0$, the constraint becomes $G_i(z) \leq M$, which is trivially satisfied by the definition of $M$.
Problem~\eqref{op: intro special case} can be formulated as the \ac{MINLP}
\begin{mini!}[2]
    {z \in \mathbb{R}^{n_z}, \delta \in \{0, 1\}^{n_\delta}}{f(z) - \sum_{i=1}^{n_\delta}w_i\delta_i}{\label{op: generic minlp}}{}
    \addConstraint{g(z)}{=0, \; h(z) \leq 0}
    \addConstraint{G_i(z)}{\leq M(1 - \delta_i),}{\quad i \in \Z_{[1, n_\delta]}.\label{cns: generic bigM constraint}}
\end{mini!}
General \acp{MINLP} are NP-hard problem, even undecidable if the respective problem is unbounded~\cite{Kannan1978}.
For the special case of \textit{convex} MINLPs, where relaxing the integer variables results in a convex \ac{NLP},  existing solvers can often compute the global optimum reasonably fast, if a feasible solution exists.
However, the complexity of \textit{convex} \acp{MINLP} remains NP-hard.
In Section~\ref{sec: solution methods}, we present an overview of methods for solving \acp{MINLP} with a special focus on a recently proposed method~\cite{Ghezzi2024}, which has demonstrated to work directly and effectively on nonconvex \acp{MINLP} when other solvers fail.

The second strategy we consider for modeling the logical expression is by introducing complementarity or vanishing constraints~\cite{Scholtes2004,Achtziger2008}.
The logical implication considered in~\eqref{cns: intro implication special case} only requires vanishing constraints~\cite{Achtziger2008}.
Specifically, the implication is substituted by the nonconvex constraints
\begin{align}\label{eq: vanishing constraints formulation}
        \delta_i G_i(z) \leq 0, \; \delta_i \in [0,1], \quad i \in \Z_{[1, n_\delta]}.
\end{align}
When $\delta_i = 0$, the constraint is trivially satisfied, but when $\delta_i > 0$ it must hold that $G_i(z) \leq 0$.
By means of~\eqref{eq: vanishing constraints formulation}, we can formulate~\eqref{op: intro special case} as the \ac{MPVC}
\begin{mini!}[2]
    {z \in \mathbb{R}^{n_z}, \delta \in [0, 1]^{n_\delta}}{f(z) - \sum_{i=1}^{n_\delta}w_i \delta_i}{\label{op: generic mpvc}}{}
    \addConstraint{g(z)}{=0, \; h(z) \leq 0}
    \addConstraint{\delta_i G_i(z)}{\leq 0,}{\; i \in \Z_{[1, n_\delta]}. \label{cns: generic vanishing constraint}}
\end{mini!}
Even though each constituent function is twice continuously differentiable, possibly convex, the vanishing constraint~\eqref{cns: generic vanishing constraint} renders~\eqref{op: generic mpvc} nonconvex and nonsmooth.
The nonsmoothness might be mitigated by using complementarity functions, denoted as ``C-functions'' or ``NCP-functions'' (cf.~\cite{Nurkanovic2023} for an overview).
However, \acp{MPVC} are hard to solve even for powerful Newton-based \ac{NLP} solvers because they violate constraint qualifications~\cite{Achtziger2008}.
In Section~\ref{sec: solution methods}, we outline an effective numerical method to solve \acp{MPVC} based on relaxation and homotopy.

\subsection{Comparison of the reformulations}
\begin{figure}
    \centering
    \includegraphics{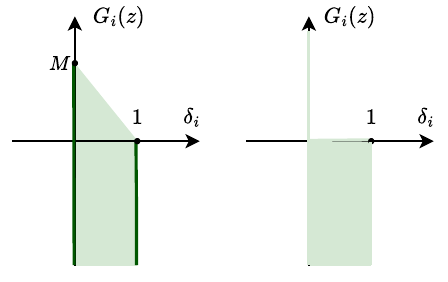}
    \caption{Left: in dark green the feasible set of constraint~\eqref{cns: generic bigM constraint} with $\delta_i \in \{0, 1\}$, and in light green its relaxation. i.e., $\delta_i \in [0, 1]$. Right: in light green the feasible set of constraint~\eqref{cns: generic vanishing constraint}.}
    \label{fig: feasible set simple minlp and mpvc}
\end{figure}

The two constraint reformulations that we have described lead to different feasible sets for the corresponding relaxations.
Figure~\ref{fig: feasible set simple minlp and mpvc} depicts the feasible set of constraints~\eqref{cns: generic bigM constraint} and~\eqref{cns: generic vanishing constraint}, respectively.
First, consider the feasible set expressed by~\eqref{cns: generic bigM constraint} for the MINLP.
In Figure~\ref{fig: feasible set simple minlp and mpvc}, we see that the feasible set is characterized by two disjoint lines, and an ambiguity in the value of the indicator variable might arise when $G_i(z)=0$.
In this case, both the alternatives
\begin{equation}
    G_i(z) \leq M, \; \delta_i=0, \qquad G_i(z) \leq 0, \; \delta_i=1,
\end{equation}
are feasible for $G_i(z)=0$.
The ambiguity is resolved by the cost function, since an optimal solution is characterized by $\delta_i=1$ (cf. Lemma~\ref{lemma: characterization optimal solution}).

A similar reasoning can be applied to the MPVC reformulation which has the connected feasible set shown in Figure~\ref{fig: feasible set simple minlp and mpvc}.
When $G_i(z)\leq 0$, the minimization of the cost causes $\delta_i=1$, resolving the possible ambiguity.
Following this intuition we can formally state this property regarding the optimal solutions of~\eqref{op: generic mpvc}.
\begin{lemma}\label{lemma: characterization optimal solution}
    Let $(z^\star, \delta^\star)$ be a locally optimal solution of the relaxed MINLP~\eqref{op: generic minlp}, i.e., with $\delta \in [0, 1]^{n_\delta}$, and let $\mathcal{I} \subseteq \Z_{[1, n_\delta]}$ be the set of indices such that $G_i(z^\star) \leq 0$ for all $i \in \mathcal{I}$, then $\delta_i^\star = 1$ for all $i \in \mathcal{I}$.
\end{lemma}
\begin{proof}
    By contradiction, suppose that $(z^\star, \delta^\star)$ is a locally optimal solution where for all $i \in \mathcal{I} \subseteq \Z_{[1, n_\delta]}, G_i(z^\star) \leq 0$, and $0< \delta_i^\star < 1$.
    There exists a feasible direction that does not change $z^\star$ but acts on $\delta^\star$ improving the objective value.
    Therefore, $\delta^\star$ with $0< \delta_i^\star < 1, i \in \mathcal{I}$ is not optimal, but $\delta^\star$ with $\delta_i^\star = 1,$ for all $i \in \mathcal{I}$ is a locally optimal solution.
\end{proof}
Lemma~\ref{lemma: characterization optimal solution} also holds for the optimal solution of the MPVC formulation~\eqref{op: generic mpvc}.
The proof follows the same argument.
Additionally, for the MPVC we can demonstrate that a locally optimal solution does not admit fractional indicator variables.
\begin{theorem}
    Let $(z^\star, \delta^\star)$ be a locally optimal solution of~\eqref{op: generic mpvc} then $\delta^\star \in \{0, 1\}^{n_\delta}$.
\end{theorem}
\begin{proof}
    In~\eqref{op: generic mpvc}, the indicator variables $\delta \in [0, 1]^{n_\delta}$ enter only constraint~\eqref{cns: generic vanishing constraint}, and only three alternatives are possible
    \begin{enumerate}
        \item $G_i(z) < 0$, then constraint~\eqref{cns: generic vanishing constraint} is satisfied for $\delta_i \in [0, 1]$. An optimal solution is characterized by $\delta_i = 1$ (cf. Lemma~\ref{lemma: characterization optimal solution}).
        \item $G_i(z) = 0$ the same reasoning applies.
        \item $G_i(z) > 0$, then constraint~\eqref{cns: generic vanishing constraint} is satisfied only for $\delta_i = 0$, since $\delta_i \in [0, 1]$.
    \end{enumerate}
    Therefore, we can conclude that a locally optimal solution of~\eqref{op: generic mpvc} admits only integer indicator variables, i.e., $\delta \in \{0, 1\}^{n_\delta}$.
\end{proof}

A different situation applies when we consider the relaxation of~\eqref{op: generic minlp}, i.e., $\hat{\delta} \in [0, 1]^{n_\delta}$.
As depicted in Figure~\ref{fig: feasible set simple minlp and mpvc}, the big-M method creates a larger feasible set which admits cases with $G_i(z)>0$ and $\hat{\delta_i} >0$.
Therefore, the optimal solution of the relaxation might exploit the enlarged feasible set, producing a solution with fractional indicator variables.
However, it is possible to postprocess such relaxed solution, and recover a feasible solution for the original MINLP problem~\eqref{op: generic minlp}.

\begin{lemma}
    Given an optimal solution of the relaxed MINLP~\eqref{op: generic minlp}, i.e., with $\hat{\delta} \in [0, 1]^{n_\delta}$, the following holds
    \begin{equation}
        \hat{\delta}_i = 1 \Leftrightarrow G_i(z) \leq 0.
    \end{equation}
\end{lemma}
\begin{lemma}\label{lemma: rounding rule}
    Given an optimal solution of the  relaxed MINLP~\eqref{op: generic minlp}, i.e., with $\hat{\delta} \in [0, 1]^{n_\delta}$, a feasible solution of the original problem~\eqref{op: generic minlp} is obtained as
    \begin{equation}
        \delta_i = \begin{cases}
            0, \quad \mathrm{if} \;\; \hat{\delta}_i<1, \\
            1, \quad \mathrm{otherwise}.
        \end{cases}
    \end{equation}
\end{lemma}
Lemma~\ref{lemma: rounding rule} gives a way to obtain feasible solutions for~\eqref{op: generic mpvc} which may be only suboptimal.

\subsection{Explicit formulations for problems with Heaviside step-function and indicator constraints}\label{sec: generalization to indicator constraints}
Next, we derive a \ac{MINLP} and \ac{MPVC} formulation that can be readily treated by a numerical solver for the general mathematical program introduced in~\eqref{op: intro}.
The main difficulty resides in the representation of the Heaviside step-function in the cost~\eqref{cf: extended nlp with logic implication}, and of the logical implication with general indicator constraints~\eqref{cns: intro implication}.
Here, for completeness, we consider the left-hand-side of constraint~\eqref{cns: intro implication} to hold with strict inequality.
Again, we propose a formulation with big-M constraints and one with vanishing constraints.
For the former, for $i\in \Z_{[1, n_\delta]}$ and for every admissible value of $z$, we assume $-m \leq H_i(z) \leq M$ with $m, M \geq 0$, and introduce binary variables $\delta_i \in \{0, 1\}$ to state the logical relation
\begin{equation}\label{eq: implication chain}
    H_i(z) > 0 \Rightarrow \delta_i = 1 \Rightarrow G_i(z) \leq 0.
\end{equation}
Expression~\eqref{eq: implication chain} can be enforced by the following inequality constraints
\begin{align}
    \begin{cases}\label{eq: indicator constraint integer formulation}
        H_i(z) \leq M \delta_i \\
        G_i(z) \leq M(1 - \delta_i).
    \end{cases}
\end{align}
Conversely, to represent~\eqref{cns: intro implication} with vanishing constraints it is not enough to enforce the constraint
\begin{equation}\label{eq: indicator constraint vanishing formulation}
    H_i(z) G_i(z) \leq 0, \quad i\in \Z_{[1, n_\delta]},
\end{equation}
because the product operation does not correspond to a logical implication, see Table~\ref{tab: truth table implication and product for mpvc}.
\begin{table}
    \centering
    \begin{tabular}{cc|c|c}\toprule
        $a$ & $b$ & $a \Rightarrow b$ & $ab\leq0$ \\ \midrule
        0 & 0 & 1 & 1\\
        0 & 1 & 1 & 0\\
        1 & 0 & 0 & 0\\
        1 & 1 & 1 & 1\\
        \bottomrule
    \end{tabular}
    \caption{Truth table for logical operations, where $a\coloneqq (H_i(z)>0)$ and $b\coloneqq(G_i(z)\leq0)$.}
    \label{tab: truth table implication and product for mpvc}
\end{table}
To correctly represent the logical implication we must require both~\eqref{eq: indicator constraint vanishing formulation} and
\begin{equation}\label{eq: H(z) >= 0}
    H_i(z) \geq 0, \quad i \in \Z_{[1, n_\delta]}.
\end{equation}
However, imposing~\eqref{eq: H(z) >= 0} unless it is already present in the original formulation, reduces the feasible set and possibly harms the quality of the solution.
The reduction of the feasible set is illustrated in Figure~\ref{fig: feasible set generic minlp and mpvc} where the big-M formulation~\eqref{eq: indicator constraint integer formulation} represents the logical implication correctly without requiring~\eqref{eq: H(z) >= 0}.
Thus, adding~\eqref{eq: H(z) >= 0} to~\eqref{eq: indicator constraint integer formulation} introduces an unnecessary restriction.
\begin{figure}
    \centering
    \includegraphics{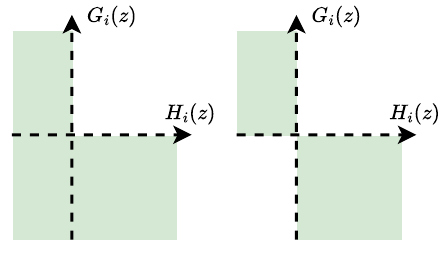}
    \caption{Left: in light green the feasible set of constraint~\eqref{eq: indicator constraint integer formulation} (the indicator variable $\delta_i$ is projected out). Right: in light green the feasible set of constraint~\eqref{eq: indicator constraint vanishing formulation}.
    In both plots the dashed lines are part of the feasible set.
    The MPVC constraints~\eqref{eq: indicator constraint vanishing formulation} do not correctly represent the logical implication. By considering~\eqref{eq: H(z) >= 0}, we obtain two identical feasible sets and a correct representation of the logical implication. However, this introduces unnecessary restrictions for the \ac{MINLP} formulation.}
    \label{fig: feasible set generic minlp and mpvc}
\end{figure}

A correct way to represent the logical implication via structured nonconvexities without restricting the feasible set involves auxiliary variables and complementarity constraints.
First, we introduce the complementarity behavior, consider the variables $x, y$, such that $0\leq x \perp y \geq 0$.
The orthogonality symbol specifies a complementarity behavior which can be made explicit in different ways as $x, y \geq 0, xy \leq 0$, or $\min(x, y) = 0$, or  $\sqrt{x^2 + y^2} - (x+y) = 0$.
Typically, the first option is the most popular.
Each of these three representations leads to different yet degenerate \acp{NLP}, since they all violate standard constraint qualification.
Let $y_i \in \R, i \in \Z_{[1, n_\delta]}$ be auxiliary variables , the feasible set represented by~\eqref{eq: indicator constraint integer formulation} can also be obtained as
\begin{align}
    \begin{cases}\label{eq: extended with mpcc}
        0 \leq y_i \perp y_i - H_i(z) \geq 0, \\
        y_i G_i(z) \leq 0.
    \end{cases}
\end{align}
Thus, the resulting problem contains both complementarity and vanishing constraints, and it can be classified as a \ac{MPCC}, since every vanishing constraint can be reformulated as a complementarity one.
Now, we have obtained two ways to reformulate the logical implication~\eqref{cns: intro implication}.
The representation of~\eqref{cns: intro implication} as~\eqref{eq: extended with mpcc} has been proposed in~\cite{Szmuk2020}, where the constraints expressed by function $G_i(z)$ are holding with equality. 

Next, we consider the representation of the Heaviside step-function in the cost~\eqref{cf: extended nlp with logic implication}.
For the big-M formulation, one can substitute $\indfn$ simply by the auxiliary indicator variables $\delta_i$.
However, when we consider~\eqref{cns: intro implication} to hold with strict inequality, i.e., ``$H_i(z)$ > 0'', optimal solutions would have an issue for $H_i(z) = 0$, as both the following cases obtained from~\eqref{eq: indicator constraint integer formulation} are feasible
\begin{align}
    \delta_i = 0:
    \begin{cases}
        H_i(z) \leq 0, \\
        G_i(z) \leq M,
    \end{cases}
    \delta_i = 1:
    \begin{cases}
        H_i(z) \leq M, \\
        G_i(z) \leq 0.
    \end{cases}
\end{align}
Based on the weighting coefficient $w_i$ in the cost~\eqref{cf: extended nlp with logic implication} one of two cases is optimal.
Indeed, the optimizer would seek a solution with $\delta_i=1$ which imposes $G_i(z) \leq 0$, even if it is not required since $H_i(z) = 0$.
Moreover, $G_i(z) \leq 0$ is harder to satisfy than $G_i(z) \leq M$.
Therefore, if the optimizer sets $\delta_i =1$, the improvement in the objective value overcomes the possible benefit of a larger feasible set.
If we want to avoid the ambiguity for $H_i(z) = 0$, we can strengthen the relation between $H_i$ and $\delta_i$
\begin{equation}
    H_i(z) > 0 \Leftrightarrow \delta_i = 1 \Rightarrow G_i(z) \leq 0,
\end{equation}
which is translated via the big-M formulation into the inequalities
\begin{align}\label{eq: eps indicator constraint integer formulation}
    \begin{cases}
        H_i(z) \leq M \delta_i \\
        H_i(z) \geq -m (1-\delta_i) + \varepsilon \\
        G_i(z) \leq M(1 - \delta_i),
    \end{cases}
\end{align}
where we modified the second inequality by adding a small positive scalar $\varepsilon \in \R_+$, in order to obtain
\begin{align}
    \delta_i = 0: H_i(z) \geq -m + \varepsilon, \quad \delta_i = 1: H_i(z) \geq \varepsilon.
\end{align}
It is evident that with this modification the solution $\delta_i = 1$ is not feasible when $H_i(z) = 0$.
Therefore, the cost term~\eqref{eq: heaviside step function} can be replaced by $\delta_i$.
Unfortunately, this modification has also reduced the feasible set of constraint~\eqref{cns: intro implication} as illustrated in Figure~\ref{fig: feasible minlp with epsilon}.
However, such change may be arbitrarily small based on $\varepsilon$.
\begin{figure}
    \centering
    \includegraphics{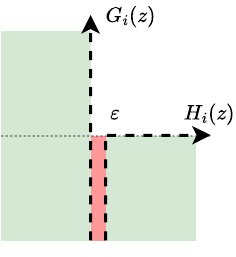}
    \caption{Feasible set of constraints~\eqref{eq: eps indicator constraint integer formulation}. The dashed lines are part of the feasible set, the red rectangle defined from zero to $\varepsilon$ is the portion of the feasible set removed.}
    \label{fig: feasible minlp with epsilon}
\end{figure}

\smallskip

We now turn our attention to the formulation with vanishing constraints.
As stated above, the logical implication can be represented by imposing~\eqref{eq: indicator constraint vanishing formulation},~\eqref{eq: H(z) >= 0}.
The unresolved issue is how to obtain the function $\indfn$ in the cost, with a representation more amenable for computations.
An option is to use a surrogate continuous function as
\begin{equation}\label{eq: sigmoid fn as indicator function}
    \approxindfn^{\mathrm{SIG}}(H_i(z)) \coloneqq \frac{1}{1 + e^{-\beta H_i(z)}}, \quad i \in \Z_{[1, n_\delta]},
\end{equation}
where the coefficient $\beta \in \R_+$ can be tuned to obtain a steeper transition between 0 and 1.
Alternatively, it is possible to represent~\eqref{eq: heaviside step function} by the \ac{KKT} conditions of the \ac{LP}
\begin{align}\label{eq: KKT representation of indicator function}
    \min_{\delta_i \in \R} \delta_i H_i(z) \; \text{ s.t. } 0 \leq \delta_i \leq 1 \Leftrightarrow
    \begin{cases}
        0 \leq \delta_i \perp \lambda_{1, i} \geq 0, \\
        0 \leq 1-\delta_i \perp \lambda_{2, i} \geq 0, \\
        H_i(z) - \lambda_{1, i} + \lambda_{2, i} = 0,
    \end{cases} \quad i \in \Z_{[1, n_\delta]},
\end{align}
where $\lambda_1, \lambda_2$ are Lagrangian multipliers associated with the constraints $0 \leq \delta_i$ and $\delta_i \leq 1$, respectively.
In this case,
\begin{align}\label{eq: indicator fn approx with KKT}
    \approxindfn^{\mathrm{KKT}}(H_i(z)) \coloneqq
    \begin{cases}
        \{0\}, & \mathrm{if} \; H_i(z) < 0, \\
        [0, 1], & \mathrm{if} \; H_i(z) = 0, \\
        \{1\}, & \mathrm{if} \; H_i(z) > 0.
    \end{cases}
\end{align}
The two different formulations of the Heaviside step-function are shown in Figure~\ref{fig: reformulation indicator function}.
Both representations are approximations of the ideal behaviour~\eqref{eq: heaviside step function}, since they cannot represent exactly the discountinuity at $H_i(z) = 0$.
While~\eqref{eq: indicator fn approx with KKT} seems closer to the desired behavior, this may come at the price of more difficult computations.

\begin{figure}
    \centering
    \includegraphics{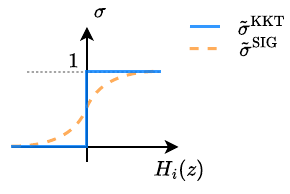}
    \caption{Reformulations of the Heaviside step-function: in dashed orange the representation via the sigmoid function~\eqref{eq: sigmoid fn as indicator function}, in solid blue the representation via \ac{KKT} conditions~\eqref{eq: KKT representation of indicator function}.}
    \label{fig: reformulation indicator function}
\end{figure}

Finally, the overall \ac{MINLP} formulation is
\begin{mini!}[2]
    {\substack{z \in \mathbb{R}^{n_z},\\\delta \in \{0, 1\}^{n_\delta}}}{f(z) - \sum_{i=1}^{n_\delta} w_i \delta_i}{\label{op: extended MINLP bigM}}{}
    \addConstraint{g(z) }{= 0, \; h(z) \leq 0}
    \addConstraint{H_i(z)}{ \leq M \delta_i,}{\quad i \in \Z_{[1, n_\delta]}}
    \addConstraint{H_i(z)}{\geq -m (1-\delta_i) + \varepsilon,}{\quad i \in \Z_{[1, n_\delta]}}
    \addConstraint{G_i(z)}{ \leq M(1 - \delta_i),}{\quad i \in \Z_{[1, n_\delta]}.}
\end{mini!}
Regarding the \ac{MPVC} formulations, the feasible set of the vanishing constraints correctly represented by~\eqref{eq: extended with mpcc} leads to formulate \acp{MPCC}.
Specifically, when the Heaviside step-function in the cost is represented by~\eqref{eq: KKT representation of indicator function}, the formulation is
\begin{mini!}[2]
    {\substack{z \in \mathbb{R}^{n_z}, \\ y_i \in [0, 1]^{n_\delta},\\ \delta, \lambda_1, \lambda_2 \in \R^{n_\delta}}}{f(z) - \sum_{i=1}^{n_\delta} w_i \delta_i}{\label{op: extended MPVC with kkt}}{}
    \addConstraint{g(z) }{= 0, \; h(z) \leq 0}
    \addConstraint{y_i G_i(z)}{\leq 0,}{\quad i \in \Z_{[1, n_\delta]}}
    \addConstraint{0 \leq y_i}{\perp y_i - H_i(z) \geq 0,}{\quad i \in \Z_{[1, n_\delta]}\label{cns: ext MPVC with kkt, complementarity 1}}
    \addConstraint{0 \leq \delta_i}{\perp \lambda_{1, i} \geq 0,}{\quad i \in \Z_{[1, n_\delta]}\label{cns: ext MPVC with kkt, complementarity 2}}
    \addConstraint{0 \leq 1-\delta_i}{\perp \lambda_{2, i} \geq 0,}{\quad i \in \Z_{[1, n_\delta]}\label{cns: ext MPVC with kkt, complementarity 3}}
    \addConstraint{H_i(z) - \lambda_{1, i} + \lambda_{2, i}}{=0,}{\quad i \in \Z_{[1, n_\delta]}.}
\end{mini!}
Although~\eqref{op: extended MPVC with kkt} does not involve integer variables, it includes complementarity constraints~\eqref{cns: ext MPVC with kkt, complementarity 1}-\eqref{cns: ext MPVC with kkt, complementarity 3}.
Complementarity constraints are akin to vanishing constraints since they violate standard constraint qualifications making the resulting \ac{NLP} hard to solve and requiring special numerical solvers.
The \ac{MPCC} reformulation with a surrogate of the Heaviside step-function, e.g.,~\eqref{eq: sigmoid fn as indicator function}, is
\begin{mini!}[2]
    {\substack{z \in \mathbb{R}^{n_z}, \\y_i \in [0, 1]^{n_\delta}}}{f(z) - \sum_{i=1}^{n_\delta} w_i \approxindfn^\textrm{SIG}(H_i(z))}{\label{op: extended MPVC with surrogate}}{}
    \addConstraint{g(z) }{= 0, \; h(z) \leq 0}
    \addConstraint{y_i G_i(z)}{\leq 0,}{\quad i \in \Z_{[1, n_\delta]}}
    \addConstraint{0 \leq y_i}{\perp y_i - H_i(z) \geq 0,}{\quad i \in \Z_{[1, n_\delta]}.}
\end{mini!}

\section{Methods for solving MINLP and MPVC}\label{sec: solution methods}
In this section, we review some of the methods for solving \acp{MINLP} and \acp{MPVC}.
The methods presented have been selected because they are general methods for their respective problem classes, they rely on existing software packages and adopt stable solvers in their subroutines, their implementations is open-source, any dependence on closed-source solvers is optional, and they have a user-friendly interface suitable for \acp{OCP}.
At the end of the section, we compare the \ac{MINLP} and the \ac{MPVC} formulations on a tutorial example based on the \ac{UGV}/\acp{UAV} coordination problem.

\subsection{Solving MINLPs}
For the solution of \acp{MINLP} we briefly present two solution methods: \ac{NBB}\cite{Dakin1965,Gupta1985} and \ac{S-B-MIQP}~\cite{Ghezzi2024}.
\ac{NBB} is a standard and well-studied solution approach that has multiple open-source and commercial implementations.
Instead, \ac{S-B-MIQP} is a recently proposed algorithm that has proven to be suitable for nonconvex MINLP arising from the time discretization of mixed-integer \ac{OCP}.

\subsubsection{Nonlinear branch-and-bound}
\Acf{NBB} is the direct extension of the branch-and-bound method for solving \ac{MILP}, introduced by~\cite{Land1960}, first presented for \acp{MINLP} in~\cite{Dakin1965}, and further studied in~\cite{Gupta1985} for convex \acp{MINLP}.
For convex \acp{MINLP}, \ac{NBB} returns the global optimum, if one exists, while for nonconvex \acp{MINLP} it only finds feasible solutions.
\ac{NBB} is often taken as baseline method for solving \acp{MINLP} since it is a general method and there exists a well-interfaced open-source implementation in the Bonmin software package~\cite{Bonami2005}.

\ac{NBB} solves the \ac{MINLP} by relaxing the integer variables and solving the continuous (convex) \ac{NLP} relaxations.
The search is typically represented by a tree in which each node is a continuous NLP to solve.
If a feasible solution of the NLP relaxation has all integer variables taking integer values, then it is also feasible (potentially globally optimal) for the MINLP.
Each continuous relaxation with some real valued integer variables is branched into two new NLP subproblems, where a new fractional integer variable is fixed to its lower and upper bound, respectively.
The solution of the node problem provide a valid lower bound that can be exploited for branching, and in case the solution is integer feasible it also provides a valid upper bound, which is useful for pruning.
In fact, if the solution of a node has an objective higher than the current upper bound, the node can be pruned from the tree.
Many ingredients are necessary for obtaining an efficient \ac{NBB} algorithm such as tight continuous relaxation, cuts to strengthen the relaxations, efficient integration of the NLP solver, and branching strategies.
For more details on \ac{NBB} see, e.g.,~\cite{Bonami2013, Floudas1995}

\subsubsection{Sequential Benders-based MIQP}
This method tackles the \ac{MINLP} solution from the point of view of decomposition methods, where the aim is to solve separately the continuous and the integer part.
The method presented in~\cite{Ghezzi2024}, which builds on ideas in~\cite{Buerger2023, Ghezzi2023a}, is developed to directly address \acp{MINLP} arising from the time discretization of \acp{OCP}, and it has shown to be competitive with state of the art solvers on existing benchmarks.
Moreover, an open-source implementation is available in the software package \texttt{CAMINO}~\cite{camino}.
\ac{S-B-MIQP} is based on a three-step procedure.
First, the problem is linearized at the current best solution (``linearization point''), which corresponds the feasible solution of the original \ac{MINLP} with the lowest objective among all points visited by the algorithm.
Second, a \ac{MIQP} with positive semidefinite Hessian approximation is constructed at the linearization point and solved.
Third, the integer solution obtained from the MIQP is fixed into the original MINLP resulting in a continuous NLP, which is then solved.
If the NLP is feasible, its solution together with the fixed integer variables is a feasible solution of the original MINLP.
If the NLP is not feasible, the algorithm switches to a feasibility NLP similarly to the outer approximation scheme proposed in~\cite{Fletcher1994}.
The solutions are used to construct new cutting planes that restrict the integer search space of the \ac{MIQP} to be solved in the next algorithm iteration.
This cutting planes are similar to the ones derived in the \ac{GBD}~\cite{Geoffrion1972}, upgraded with a regularization method proposed in~\cite{Kronqvist2020}.
Solving exclusively \acp{MIQP} in a decomposition scheme does not guarantee termination with a global optimum in case of convex MINLP, as already shown in~\cite{Fletcher1994}, because the solution of a \ac{MIQP} does not provide a valid lower bound for the solution of the original \ac{MINLP}.
Therefore, \ac{S-B-MIQP} solves a tailored \ac{MILP} whenever the solution of the \ac{MIQP} stagnates during the S-B-MIQP iterations.
The \ac{MILP} constructed in \ac{S-B-MIQP} is similar to the one adopted in \ac{GBD} with additional linear constraints resulting from the linearization of the constraints in the original \ac{MINLP} about the current best point.

\subsection{Solving MPVCs}\label{sec: solution methods for mpvc}

\acp{MPVC} can be reformulated equivalently as \acp{MPCC}, and therefore, share similar solution techniques.
Crucially, a solver method for \acp{MPVC} needs a custom way to deal with the structured nonconvexity.
As already shown in~\cite{Hoheisel2008} and~\cite[\S 9.3]{Kirches2011}, adopting a generic NLP solver for \ac{MPVC} would often result in convergence issues due to the lack of constraint qualifications.
Two distinct methods exist to treat nonconvexity, one is based on an active set method and one on a relaxation method.
To the authors' knowledge an active set-based solver for \ac{MPVC} is implemented and tested only in~\cite{Kirches2011}, but no public implementation is available.
Conversely, relaxation methods rely on generic \ac{NLP} solvers and on a homotopy loop which can be quickly implemented.
In~\cite{Nurkanovic2023}, it is shown that even simple relaxation methods are effective for solving a large set of \ac{OCP} with complementarity constraints.
The relaxation approach adopted here has been introduced by Scholtes~\cite{Scholtes2001} for solving \acp{MPCC}, and later presented for \acp{MPVC} by~\cite[\S 10]{Hoheisel2009}.
An algorithm similar to the one utilized in this work has been adopted in~\cite{Jung2013, Meyer2018}.
We introduce a vector $\boldsymbol{\tau} \in \R^{n_{G_i}}$ such that $\boldsymbol{\tau} \coloneqq (\tau, \dots, \tau)$ that relaxes the vanishing constraint into
\begin{equation}\label{cns: relaxed vanishing constraint}
    G_i(z) H_i(z) \leq \boldsymbol{\tau}, \quad i \in \Z_{[1, n_\delta]}.
\end{equation}
Then, we solve~\eqref{op: generic mpvc} with the relaxed constraint~\eqref{cns: relaxed vanishing constraint} within an homotopy loop as presented in Algorithm~\ref{alg:homotopy}.
In the numerical simulations the parameters of Algorithm~\ref{alg:homotopy} are set as follows: $\tau_0 = 10^2$, $\varepsilon_0=0.6$, $\tau_{\text{min}} = 10^{-3}$, $\kappa_0 = 1.6$, $\kappa_1 = 1.2$.
\begin{algorithm}
    \caption{Homotopy method for the solution of \acp{MPVC}}
    \label{alg:homotopy}
    \begin{algorithmic}[1]
    \Require Initial guess $z^\star$, $\tau = \tau_0 > 1$, $\varepsilon=\varepsilon_0 < 1$, $\tau_{\text{min}} < 1$, $\kappa_0 > 1$, $\kappa_1 > 1$
    \While{$\tau^\star > \tau_{\text{min}}$}
        \State Set $\tau = \varepsilon \cdot \tau^\star$
        \State Solve problem with corresponding $\tau$ starting from last solution $z^\star$
        \If {problem is locally infeasible}:
        \State $\varepsilon = \kappa_0 \cdot \varepsilon$
        \Else \State Store solution as $z^\star$, $\tau^\star = \tau$, $\varepsilon = \varepsilon/ \kappa_1$
        \EndIf
    \EndWhile
    \State \textbf{return} $z^*$
    \end{algorithmic}
\end{algorithm}

\subsection{Tutorial example: UGV/UAVs coordination problem}
In the following, we present an example illustrating how to formulate and solve the MINLP and the MPVC versions of the UGV/UAVs coordination problem introduced in Section~\ref{sec: trajectory planning problems}.
For modeling the UGV, we adopt a standard single-track kinematic model where the state is $x(t)=(p_x(t), p_y(t), \theta(t), v(t), \phi(t))$, and the control $u(t) = (a(t), \psi(t))$.
Position of the center of the rear axle along the axes $x, y$ is denoted with $p_x, p_y$, respectively, $\theta$ is the heading angle, $v$ is the velocity, $\phi$ is the steering angle, $a$ is the acceleration, and $\psi$ is the steering angular rate.
The dynamics are described by the \ac{ODE}
\begin{align}\label{eq: ugv ode}
    \dot{x} =
    \left( \begin{matrix} \dot{p}_x \\ \dot{p}_y \\ \dot{\theta} \\ \dot{v} \\ \dot{\phi} \end{matrix} \right) =
    \begin{cases}
        v \cos(\theta), \\ v \sin(\theta), \\ \dfrac{v \tan(\phi)}{L}, \\ a, \\ \psi.
    \end{cases}
\end{align}
State and control vectors are subject to constraints,
\begin{align}\label{eq: box constraint ugv}
    \mathcal{X} &= \{x \in \R^5 \;| \; |\theta| \leq \theta_\mathrm{max}, v \in [v_\mathrm{min}, v_\mathrm{max}], |\phi| \leq \phi_\mathrm{max} \}, \\
    \mathcal{U} &= \{u \in \R^2 \; | \; |a|\leq a_\mathrm{max}, |\psi| \leq \psi_\mathrm{max} \}.
\end{align}
With a slight abuse of notation we denote with $p$ the concatenation of the position along the two axes, $p = (p_x, p_y)$.
For simplicity here, the reachable sets of the monitoring targets are expressed as rectangles in the position space as
\begin{equation}
    A_i p + b_i \leq 0, \quad A_i \in \R^{4\times 4}, b_i \in \R^4, \quad i \in \Z_{[1, n_\mathrm{p}]},
\end{equation}
and they are depicted in Figure~\ref{fig: position comparison ugv} with pink rectangles, and $n_\mathrm{p}=5$.

As described earlier we aim at planning a time-optimal trajectory for the \ac{UGV} that goes from the start to the end point while visiting at least once each monitoring targets and fulfilling constraints.
To avoid a tuning effort beyond the scope of this illustrative example, we simply fix the final time $t_\mathrm{f} = 76$ seconds.
We formulate an \ac{OCP} to model such trajectory planning problem, and we solve it by using direct multiple shooting~\cite{Bock1984}.
Here, we specialize the OCP template~\eqref{op: OCP-NLP template} as
\begin{mini!}
    {\mathbf{x},\mathbf{u}, \boldsymbol{\delta}}{w\delta_N + \sum_{k=0}^{N-1} \left( \norm{u_k}^2 + w \sum_{i=1}^{n_\mathrm{p}} \delta_{k, i} \right)}{\label{op: ugv OCP}}{}
    \addConstraint{x_0}{= \bar{x}_0}
    \addConstraint{x_{k+1}}{= F(x_k, u_k, h),}{\quad k \in \Z_{[0, N-1]}}
    \addConstraint{u_k}{\in \mathcal{U},}{\quad k \in \Z_{[0, N-1]}}
    \addConstraint{x_k}{\in \mathcal{X},}{\quad k \in \Z_{[0, N]}}
    \addConstraint{\delta_{k, i} = 1}{\Rightarrow A_{i} p_k + b_{i} \leq 0,}{\quad k \in \Z_{[0, N-1]}, i \in \Z_{[1, n_\mathrm{p}]} \label{cns: logic cns ugv}}
    \addConstraint{\sum_{k=0}^{N} \delta_{k, i}}{\geq 1,}{\quad i \in \Z_{[1, n_\mathrm{p}]} \label{cns: at least one delta ugv}}
    \addConstraint{\delta_{k, i}}{\in [0, 1],}{\quad k \in \Z_{[0, N-1]}, i \in \Z_{[1, n_\mathrm{p}],} \label{cns: delta type}}
\end{mini!}
where $\boldsymbol{\delta} \coloneqq (\delta_{0, 1}, \dots, \delta_{0, n_\mathrm{p}}, \dots, \delta_{N, 1}, \dots, \delta_{N, n_\mathrm{p}})$.
Function $F$ corresponds to the discretization of~\eqref{eq: ugv ode} with a 4-step explicit Runge-Kutta integrator with sampling time $t_\mathrm{d} = t_\mathrm{f} / N$.
Also, constraints~\eqref{eq: box constraint ugv} are discretized and imposed only at the grid nodes.
Constraints~\eqref{eq: ugv logic cns generic 1},~\eqref{eq: ugv logic cns generic 2} introduced earlier are included in~\eqref{op: ugv OCP} as~\eqref{cns: logic cns ugv},~\eqref{cns: at least one delta ugv}.

For the MINLP formulation, we modify~\eqref{cns: delta type} such that the indicator variables are binaries, thus $\delta_{k, i} \in \{0, 1\}, k \in \Z_{[0, N-1]}, i \in \Z_{[1, n_\mathrm{p}]}$.
Moreover,~\eqref{cns: logic cns ugv} is represented via big-M constraints as in~\eqref{op: generic minlp} resulting in
\begin{align}
    A_i p_k + b_i \leq M (1-\delta_{k, i}), \quad k \in \Z_{[0, N-1]}, i \in \Z_{[1, n_\mathrm{p}]}.
\end{align}
For the MPVC formulation,~\eqref{cns: logic cns ugv} is represented as in~\eqref{op: generic mpvc} resulting in
\begin{align}\label{eq: vanishing cns tutorial example}
    \delta_{k, i} (A_i p_k + b_i) \leq 0, \quad k \in \Z_{[0, N-1]}, i \in \Z_{[1, n_\mathrm{p}]}.
\end{align}
The parameters used in the simulation are contained in Table~\ref{tab: params ugv}.
The two approaches share the same initial guess.
\begin{table}
    \centering
    \begin{tabular}{lrl}\toprule
        Parameter & Value & Unit \\ \midrule
        $L$ & 0.1 & m\\
        $a_\mathrm{max}$                                & 0.05 & m/s$^2$ \\
        $\psi_\mathrm{max}$                                & 0.5 & deg/s \\
        $\theta_\mathrm{max}$                                & 175 & deg \\
        $(v_\mathrm{min}, v_\mathrm{max})$                                & (0.1, 0.8) & m/s\\
        $\phi_\mathrm{max}$                                & 5 & deg \\
        $t_\mathrm{d}$                    & 3.8 & s \\
        $N$                     & 20 & - \\
        $\bar{x}_0$ & (0, 0, 0, 0.15, 0) & -\\
        $\bar{x}_N$ & (10, 10, 0, 0.15, 0) & - \\
        $w$ & -38 & - \\
        $M$ & 12 & - \\
        $\tau_\mathrm{min}$ & $10^{-4}$ & - \\
        \bottomrule
    \end{tabular}
    \caption{Parameters of~\eqref{op: ugv OCP} used in the simulations.}
    \label{tab: params ugv}
\end{table}
We solve the MINLP using the S-B-MIQP algorithm~\cite{Ghezzi2024} implemented in \texttt{CAMINO}~\cite{camino}, with \texttt{Gurobi} v10.0.3~\cite{Gurobi} as \ac{MIQP} solver, and \texttt{IPOPT} v3.14.11~\cite{Waechter2006} as \ac{NLP} solver.
The MPVC is solved with the homotopy method described in Algorithm~\ref{alg:homotopy} using \texttt{IPOPT} as NLP solver.
The example is coded in Python, the homotopy loop is also coded in Python, and \texttt{Gurobi} is used with Presolve disabled.
We use \texttt{CasADi}~\cite{Andersson2019} to model both the MINLP and the MPVC problem, and to interface the required solvers.
The algorithms within \texttt{CAMINO} adopt the same interface of \texttt{CasADi}, therefore using S-B-MIQP does not require additional interfacing work compared to Alg.~\ref{alg:homotopy}.
The code executes as a single thread on a desktop machine running Ubuntu 22.04 with an Intel(R) Core(TM) i9-13900K CPU and 128GB of memory.
Figure~\ref{fig: position comparison ugv} compares the position trajectories obtained with the two formulations, and Figure~\ref{fig: indicator comparison ugv} compares the values of the indicator variables along the two trajectories for each monitoring targets.
Table~\ref{tab: cost comparison of ugv} compares the constituent components of the cost for both formulations.
While the two formulations share the same parameters, they converge to two different solutions.
Specifically, the MINLP solution spends more time within the targets at the cost of a slightly higher control actuation compared to the MPVC solution.
The overall objective achieved by the MINLP solution is lower compared to the objective of the MPVC solution.
It is possible to achieve similar objective value with the MPVC approach by tuning the weight $w$.
Table~\ref{tab: cost comparison of ugv} reports also the computation time for the two formulations.
For this relative simple problem the runtime for solving the MINLP is much higher than that of the MPVC.
However, allowing multithreading computation for the \acp{MIQP} can reduce runtime dramatically.
On our machine, allowing \texttt{Gurobi} to use up to 32 threads reduces the runtime about 6.5 times, achieving a runtime of 4.71 seconds.
We highlight that both algorithms -- S-B-MIQP for MINLP and homotopy for MPVC -- are general purpose methods for their respective problem class.
Thus, they work out of the box without extensive tuning.
\begin{table}
    \centering
    \begin{tabular}{lrr}\toprule
        \multirow{2}{*}{Performance} & \multicolumn{2}{c}{Formulation}  \\ \cmidrule{2-3}
        measure & MINLP     & MPVC \\ \midrule
        $\sum_{k=0}^{N-1} \norm{u}_2^2$        & 0.028     & 0.025 \\
        $\sum_{k=0}^{N} \sum_{i=1}^{n_\mathrm{p}} \delta_{k, i}$    & 16         & 13 \\
        Objective & -607.97 & -493.98 \\ \midrule
        Runtime [s] & 31.00 & 1.78 \\
        \bottomrule
    \end{tabular}
    \caption{Objective breakdown and runtime of MINLP and MPVC for~\eqref{op: ugv OCP}.}
    \label{tab: cost comparison of ugv}
\end{table}
\begin{figure}
    \centering
    \subfigure{\includegraphics[width=0.35\textwidth]{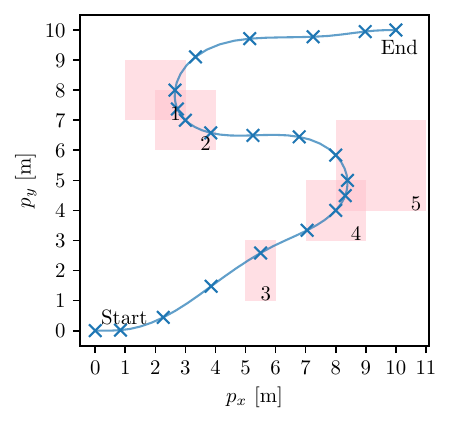}}
    \subfigure{\includegraphics[width=0.35\textwidth]{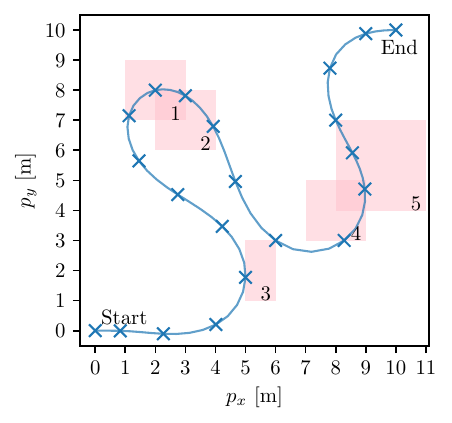}}
    \caption{Locally optimal UGV position trajectories for MINLP (left) and MPVC (right) for~\eqref{op: ugv OCP}.\label{fig: position comparison ugv}}
\end{figure}
\begin{figure}
    \centering
    \subfigure{\includegraphics[width=0.4\textwidth]{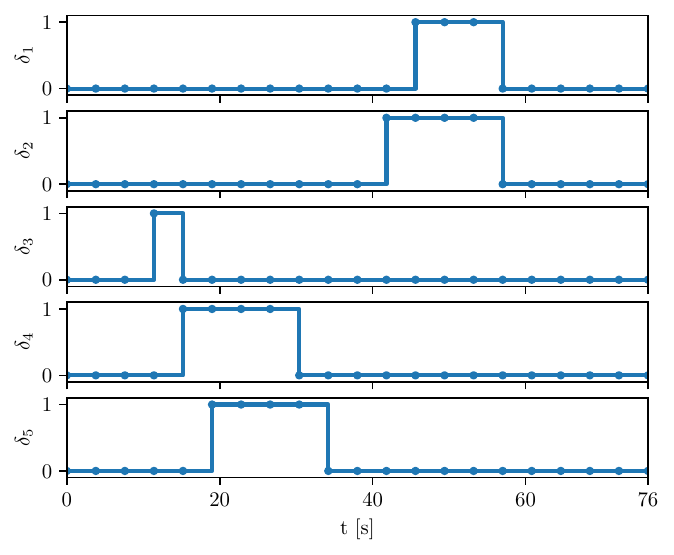}}
    \subfigure{\includegraphics[width=0.4\textwidth]{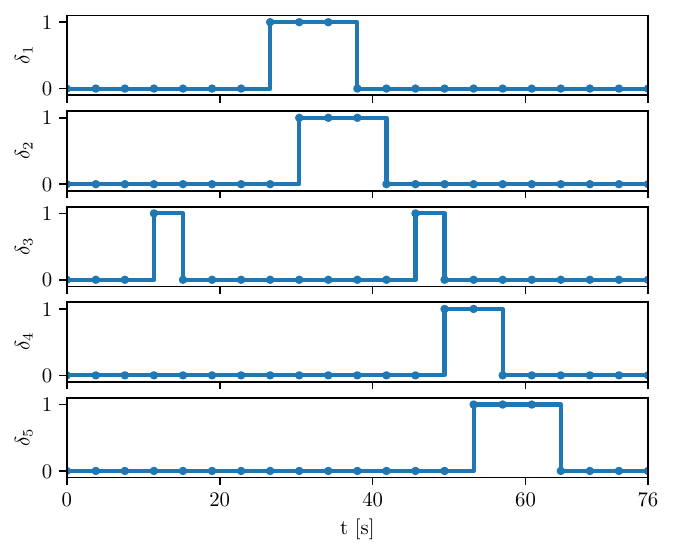}}
    \caption{Locally optimal trajectories of indicator variables for MINLP (left) and MPVC (right) for~\eqref{op: ugv OCP}.\label{fig: indicator comparison ugv}}
\end{figure}

\begin{remark}\label{remark: modelling choices}
    We want to emphasize some aspects behind our modeling choices.
    First, in the considered problems while the dynamical system has continuous state and control, the integer part is related only to the indicator variables.
    Therefore, there is no need of specific reformulations that are typically adopted in mixed-integer optimal control such as the partial outer convexification~\cite{Sager2011a}.

    Although the big-M formulation adopted for the logical constraints generally produces weaker relaxations compared to a convex hull formulation, it has the advantage of not introducing auxiliary variables and constraints.
    This is a relevant aspect since \ac{OCP} over long horizons already have large dimensions.
    For a similar reason, we do not introduce an auxiliary variable for simplifying the vanishing constraint~\eqref{eq: vanishing cns tutorial example}.

    Finally, we consider polytopes represented via halfspaces instead of vertices to avoid equality constraints, which are generally difficult to treat in \acp{MINLP}, and to avoid introducing additional variables for representing the coefficients of the vertices' convex combination.
    As a drawback, with the halfspace representation we have to impose several linear inequalities for each polytope.
    These inequalities can be tackled easily by a \ac{MIP} solver, especially if it has an effective presolve routine to eliminate redundant constraints and tighten variable bounds~\cite{Achterberg2020}.
    For a \ac{NLP} solver, a large amount of linear inequalities does not usually introduce challenges for convergence but rather for memory allocation and time spent executing linear algebra routines, for instance for matrix factorizations.
\end{remark}

\section{Case study: Divert-feasible powered descent guidance}\label{sec: landing results}
We now analyze the problem of \ac{PDG} with divert-feasible regions for Mars landing, and present formulations, solution methods and simulations for both the \ac{MINLP} and \ac{MPVC} approach.
Additionally, by means of a detailed comparison, we demonstrate computation time and objective value efficiency for both methods, when considering a problem instance close to a real application.

\subsection{Modeling}
We consider the Mars landing problem described in~\cite[p. 87]{Malyuta2022}.
The lander dynamics corresponds to a double integrator with variable mass, and it is described by the following nonlinear \ac{ODE}
\begin{align}\label{eq: ode lander}
    \begin{cases}
        \dot{r}(t) = v(t), \\
        \dot{v}(t) = g_\mathrm{mars} \hat{e}_z  + \frac{u(t)}{m(t)} - \omega^\times \omega^\times r(t) - 2 \omega^\times v(t),\\
        \dot{m}(t) = \shortminus \frac{\norm{u(t)}_2}{g_\mathrm{earth}I_\mathrm{sp}},
    \end{cases}
\end{align}
where $g_\mathrm{mars} \in \R$ is the constant gravitational acceleration of Mars, $\omega \in \R^3$ is Mars' constant angular velocity, $^\times$ denotes the vector cross-product, $g_\mathrm{earth} \in \R$ is the constant gravitational acceleration of Earth, and $I_\mathrm{sp}$ is the rocket's engine specific impulse.
The value of model parameters are reported in Table~\ref{tab: params}.
The state comprises the Cartesian position and velocity along axes $xyz$, denoted by $r(t)$ and $v(t)$, respectively, and the mass of the lander, denoted by $m(t)$.
Additionally, we assume that the lander is moving in a constant gravitational field and is viewed in the planet's rotating frame.
Drag forces are neglected as Mars' atmosphere has low density.
The \ac{ODE} can be written compactly as
\begin{equation*}
    \dot{x}(t) = f(x(t), u(t)), \quad x(t)=(r(t), v(t), m(t)) \in \R^7, \; u(t) \in \R^3.
\end{equation*}
In what follows, we drop the dependency on time for simplicity of notation.
The control $u$ is the thrust that can be produced by the lander along axes $xyz$, respectively.
State vector and control vector are constrained to lie in the sets $\mathcal{X}$ and $\mathcal{U}$, respectively,
\begin{align}
    \mathcal{X} &= \{ x \in \R^7 \; | \;  \hat{e}_z^\top r \geq \norm{r}_2 \cos(\gamma_\mathrm{gs})\}, \label{eq: state constraints on mars lander} \\
    \mathcal{U} &= \{ u \in \R^3 \; | \; 0 < \rho_\mathrm{lb} \leq \norm{u}_2 \leq \rho_\mathrm{ub}, \quad \hat{e}_z^\top u \geq \norm{u}_2 \cos(\gamma_\mathrm{p})\}. \label{eq: control constraints on mars lander}
\end{align}
A glide-slope constraint~\eqref{eq: state constraints on mars lander} is imposed to limit the approaching angle of the spacecraft with respect to the landing site.
The control constraints~\eqref{eq: control constraints on mars lander} limit the total thrust and the so-called ``pointing angle'' of the spacecraft.
The left-hand-side of the thrust norm makes constraint~\eqref{eq: control constraints on mars lander} nonconvex.
For the considered 3-DoF model, the pointing angle of the spacecraft is approximated based on the ratio between the vertical thrust and the total thrust.
Limiting the pointing angle is necessary to obtain trajectories that well approximate ones computed with more sophisticated dynamical models.
Table~\ref{tab: params} contains the values of the model and constraint parameters.
\begin{table}
    \centering
    \begin{tabular}{lrl}\toprule
        Parameter & Value & Unit \\ \midrule
        $\gamma_\mathrm{gs}$                    & 86 & deg \\
        $\gamma_\mathrm{p}$                     & 40 & deg \\\
        $\omega$                                & $10^{-3} \cdot [3.5, 0, 2]^\top$ & $1/\mathrm{s}$\\
        $g_\mathrm{mars}$                       & -3.71 & $\mathrm{m/s^2}$ \\
        $g_\mathrm{earth}$                      & 9.807 & $\mathrm{m/s^2}$ \\
        $I_\mathrm{sp}$                         & 225 & $\mathrm{s}$ \\
        $\hat{e}_z$                             & $[0, 0, 1]^\top$  & $\shortminus$ \\
        $(\rho_\mathrm{lb}, \rho_\mathrm{ub})$  & $(4971, 13258)$& $\mathrm{N}$ \\
        $m_\mathrm{wet}$                       & 1905 & kg \\
        $m_\mathrm{dry}$                        & 1505 & kg \\
        \bottomrule
    \end{tabular}
    \caption{Model and constraint parameters in spacecraft landing simulations.}
    \label{tab: params}
\end{table}

\subsection{OCP formulations}
In this subsection, we provide a detailed formulation of the \ac{OCP}.
As introduced in Sec.~\ref{sec: trajectory planning problems}, divert-feasible regions are represented as polytopes.
Our goal is to compute trajectories that remain within these polytopes for as long as possible while balancing with the minimization of fuel consumption.

\subsubsection{MINLP formulation}
We can formulate the following \ac{MINLP} via direct multiple shooting~\cite{Bock1984} by discretizing the spacecraft dynamics and constraints, and by adding divert-feasible regions
\begin{mini!}
    {t_\mathrm{f}, \mathbf{u}, \mathbf{x}, \boldsymbol{\delta}}{-w_0 m_N - w_1\sum_{i=1}^{n_\mathrm{p}}  \delta_{i, k}}{\label{op: minlp lander}}{\label{cfn: minlp lander cost fn}}
    \addConstraint{x_0}{= (\bar{r}_0, \bar{v}_0, \bar{m}_0) \label{op: minlp case study cns 1}}
    \addConstraint{x_{k+1}}{=F(x_k, u_k, \tfrac{t_\mathrm{f}}{N}),}{\quad k \in \Z_{[0, N-1]} \label{op: minlp case study cns 2}}
    \addConstraint{\rho_\mathrm{lb}}{\leq u_k \leq \rho_\mathrm{ub},}{\quad k \in \Z_{[0, N-1]}\label{op: minlp case study cns 3}}
    \addConstraint{\rho_\mathrm{lb}}{\leq \norm{u_k}_2 \leq \rho_\mathrm{ub},}{\quad k \in \Z_{[0, N-1]}\label{op: minlp case study cns 4}}
    \addConstraint{\cos(\gamma_\mathrm{p}) \norm{u_k}_2}{\leq \hat{e}_z^\top u_k,}{\quad k \in \Z_{[0, N-1]}\label{op: minlp case study cns 5}}
    \addConstraint{\cos(\gamma_\mathrm{gs}) \norm{r_k}_2}{\leq \hat{e}_z^\top r_k,}{\quad k \in \Z_{[0, N-1]}\label{op: minlp case study cns 6}}
    \addConstraint{r_N = 0, v_N}{= 0, m_N \geq m_\mathrm{dry}\label{op: minlp case study cns 7}}
    \addConstraint{A_i \xi_{k, i} - b_i}{\leq M_i (1 - \delta_{k, i}),}{\quad k \in \Z_{[0, N]}, i \in \Z_{[1, n_\mathrm{p}]} \label{cns: minlp lander polytope}}
    \addConstraint{\delta_{k, i}}{\in \{0, 1\},}{\quad k \in \Z_{[0, N]}, i \in \Z_{[1, n_\mathrm{p}]},}
\end{mini!}
where $F$ is the discretization of the associated \ac{ODE} via a 1-step explicit Runge-Kutta integrator of order 4.
The bold letters $\mathbf{x}\coloneqq(x_0, \dots, x_{N})$, $\mathbf{u}=(u_0, \dots, u_{N-1})$, $\boldsymbol{\delta} = (\delta_{0, 0}, \dots, \delta_{n_\mathrm{p}, 0}, \dots, \delta_{0, N}, \dots, \delta_{n_\mathrm{p}, N})$ define the stage-wise concatenation of state, continuous control and binary indicator variables, respectively.
Remember that $\xi_k$ in constraint~\eqref{cns: minlp lander polytope} is defined as $\xi_k \coloneqq (r_k, v_k)$.
The scalar $M_i \in \R_+$ is a valid upperbound for the left-hand side of the corresponding constraint.

\subsubsection{MPVC formulation}
Again, via direct multiple shooting, we formulate the \ac{OCP}, this time by introducing vanishing constraints and indicator variables, and obtain the \ac{NLP}
\begin{mini!}
    {h, \mathbf{u}, \mathbf{x}, \boldsymbol{\delta}}{-m_N - \frac{t_\mathrm{f}}{N} \sum_{k=0}^{N}\sum_{i=1}^{n_\mathrm{p}}  \delta_{i, k}}{\label{op: mpvc lander}}{\label{cfn: mpvc lander cost fn}}
    \addConstraint{\eqref{op: minlp case study cns 1},~\eqref{op: minlp case study cns 2},~\eqref{op: minlp case study cns 3},}{\;\eqref{op: minlp case study cns 4},~\eqref{op: minlp case study cns 5},~\eqref{op: minlp case study cns 6},~\eqref{op: minlp case study cns 7} \nonumber}
    \addConstraint{\delta_{k, i} (A_i \xi_{k, i} - b_i)}{\leq \tau,}{\quad k \in \Z_{[0, N]}, i \in \Z_{[1, n_\mathrm{p}]} \label{cns: mpvc lander polytope}}
    \addConstraint{\delta_{k, i}}{\in [0, 1],}{\quad k \in \Z_{[0, N]}, i \in \Z_{[1, n_\mathrm{p}]},}
\end{mini!}
where $\tau \geq 0$ is the homotopy parameter.
For $\tau = 0$ the \ac{MPVC}~\eqref{op: mpvc lander} shares the same minimizers of the \ac{MINLP}~\eqref{op: minlp lander}.
Of course, even if~\eqref{op: minlp lander},~\eqref{op: mpvc lander} share the same parameters and initialization, they might converge to different local optima since the problems are nonconvex, and the methods adopted to solve them perform different operations to find local optima.

\subsubsection{Additional modifications to both formulations}
Regarding the discretization of the \ac{ODE}, we have parameterized the controls with piecewise linear and continuous functions.
This is a common choice for \ac{PDG} problems since it has proven to provide enough numerical accuracy without compromising computation time~\cite{Scharf2017}.
Hence, in practice, we augment~\eqref{eq: ode lander} with an integrator to represent the current thrust and create a new control representing the thrust increment.
The augmented \ac{ODE} is given by
\begin{align*}
        \begin{cases}
       ~\eqref{eq: ode lander}, \\
        \dot{u}(t) = \mu(t),
    \end{cases}
\end{align*}
such that
\begin{equation*}
    \dot{\tilde{x}}(t) = \tilde{f}(\tilde{x}(t), \mu(t)), \; \mathrm{with} \; \tilde{x}(t) = (r(t), v(t), m(t), u(t)) \in \R^{10}, \mu(t) \in \R^3.
\end{equation*}
The incremental thrust $\mu(t)$ is unbounded.
However, we add a penalization term to the cost function to obtain smoother activation profiles.
Therefore,~\eqref{cfn: minlp lander cost fn} is updated as follows
\begin{equation*}
    -w_0 m_N - w_1 \sum_{k=0}^{N}\sum_{i=1}^{n_\mathrm{p}}  \delta_{i, k} + w_2\sum_{k=0}^{N-1} \norm{\mu_k}_2^2.
\end{equation*}
Besides the integer variables in~\eqref{op: minlp lander} and the nonconvex vanishing constraints in~\eqref{op: mpvc lander}, both problems have a nonlinear dynamics~\eqref{op: minlp case study cns 2} and a nonconvex constraint, since the total thrust is lower bounded~\eqref{op: minlp case study cns 4}.
Thus,~\eqref{op: minlp lander} and~\eqref{op: mpvc lander} are challenging to solve.
To obtain computationally tractable problems, we consider some further simplifications.
First, we do not optimize for final time, and we fix it to 75 seconds.
This value corresponds to the minimum time and fuel optimal trajectory that is obtained by solving the standard \ac{PDG} problem in~\cite[p. 87]{Malyuta2022}.
Second, we add slack variables to soften the terminal state constraints on position and velocity to improve numerical stability, and we penalize the use of slacks by a squared euclidean norm term in the cost function.
Constraint~\eqref{op: minlp case study cns 7} is modified as
\begin{align*}
    r_N - s_{r, N} = 0, \; v_N - s_{v, N} = 0, \quad s_{r, N}, s_{v, N} \in \R^3,
\end{align*}
where $s_{r, N}, s_{v, N}$ slack the final position and velocity, respectively.
Also, $s_{r, N}, s_{v, N}$ are bounded such that
\begin{align*}
    \begin{bmatrix} -5 \\ -5 \\ 0   \end{bmatrix} \leq  s_{p, N} \leq \begin{bmatrix} 5 \\ 5 \\ 5   \end{bmatrix},   \norm{s_{v, N}} \leq \begin{bmatrix} 0.01 \\ 0.01 \\ 0.01  \end{bmatrix}.
\end{align*}
In this way, we guarantee that any feasible solution has a maximal deviation from the prescribed landing site of 5 meters along each coordinate, and a maximal final velocity of 1 cm/s along each coordinate.
As a result, the cost functions~\eqref{cfn: minlp lander cost fn},~\eqref{cfn: mpvc lander cost fn} are modified into
\begin{equation*}
    -w_0 m_N - w_1 \sum_{k=0}^{N}\sum_{i=1}^{n_\mathrm{p}}  \delta_{i, k} + w_2\sum_{k=0}^{N-1} \norm{\mu_k}_2^2 + \norm{s_{p, N}}_2^2 + \norm{s_{v, N}}_2^2.
\end{equation*}
the cost weights adopted in these simulations are reported in Table~\ref{tab: sim params}.

\begin{table}
    \centering
    \begin{tabular}{lrl}\toprule
        Parameter & Value & Unit \\ \midrule
        $\bar{r}_0$                             & $[2000, 0, 1500]^\top$ & m \\
        $\bar{v}_0$                             & $1/3.6 \cdot [288, 108, -270]^\top$ & m/s \\
        $\bar{m}_0$                             & $m_\mathrm{wet}$ & kg \\
        $v_\mathrm{max}$                        & $500 / 3.6$ & m/s \\
        $\alpha$                                & $(g_\mathrm{earth}I_\mathrm{sp})^{-1}$ & s/m \\
        $N$                                     & 50 & $\shortminus$ \\
        $t_\mathrm{f}$                          & 75 & s \\
        $(w_0, w_1, w_2)$                          & $(10^{-3}, 10^3, 10^{-3})$ & $\shortminus$ \\
        $\tau_\mathrm{min}$                     & $10^{-3}$ & $\shortminus$ \\
        \bottomrule
    \end{tabular}
    \caption{\ac{OCP} parameters used in the numerical simulations.}
    \label{tab: sim params}
\end{table}

\begin{remark}
    If we disregard divert-feasible regions,~\eqref{op: minlp lander} is a standard \acf{PDG} problem~\cite{Acikmese2007}.
    \ac{PDG} is a fuel- and time-optimal problem for which convex reformulations exist.
    Specifically, the standard PDG problem can be formulated as a \ac{SOCP} and solved to global optimality in the order of milliseconds.
\end{remark}

\subsection{Simulations with simplified divert feasible regions}

We consider a landing scenario with simplified divert feasible polytopic regions.
Specifically, we consider three polytopes only in the position space, i.e., polyhedrons, with a inverted-pyramid shape.\\
The polyhedrons are defined as
\begin{equation}\label{eq: pyramid equation}
    C (r_k - c_i) + d \leq 0, \quad i = 1, 2, 3, k \in \Z_{[0, N]},
\end{equation}
where
\begin{align*}
    C = \begin{bmatrix}
        \cos(\beta) & 0 & -\sin(\beta) \\
        0 & \cos(\beta) &  -\sin(\beta) \\
        -\cos(\beta) & 0 & -\sin(\beta) \\
        0 & -\cos(\beta) & -\sin(\beta)
    \end{bmatrix}, \beta = 70^\circ, d = \begin{bmatrix}
        1 \\ 1 \\ 1 \\ 1
    \end{bmatrix}, \\
    c_1 = (2000, 400, 0), c_2 = (1000, 250, 0), c_3 = (100, -100, 0).
\end{align*}
One can easily transform~\eqref{eq: pyramid equation} into the canonical linear inequality $A r_k + b_i \leq 0, \; i = 1, 2, 3,$ and $k \in \Z_{[0, N]}$ and devise constraints~\eqref{cns: minlp lander polytope} and~\eqref{cns: mpvc lander polytope}.
Thus, the MINLP and MPVC formulations can be passed to their respective solvers.
Table~\ref{tab: sim params} contains the values of the parameters used in simulation.
For Bonmin, we selected the following options that achieved faster runtime \texttt{variable selection: osi-simple}, \texttt{tree search strategy: top-node}, \texttt{node comparison: dynamic}.
For S-B-MIQP, we used the default values set in \texttt{CAMINO}.
For Alg.~\ref{alg:homotopy} we chose a $\tau_\mathrm{min} = 10^{-3}$ which is sufficient to obtain locally optimal solution with binary values for the indicator variables, the other parameters are unchanged.

Computation time and objective value for each solver are reported in Table~\ref{tab: results for simplified regions}.
In the comparison, we also add the solution from the standard \ac{PDG} problem, formulated as a \ac{SOCP} and, here, solved with IPOPT.
For the \ac{PDG} solution we omitted the objective value since it has a different cost function.
Also, the reported runtime could be reduced by choosing a tailored solver or by code generation.
We compute the optimal trajectory for the \ac{PDG} problem in order to have a baseline to compare against for the divert-feasible trajectories.
Among the solutions to the landing with divert-feasible regions, the lowest objective is achieved by the MINLP formulation solved by Bonmin.
The solution exploits the softened terminal constraints on position and velocity.
Moreover, the trajectory stays within one or multiple regions for 27 time steps, and employs roughly 3 kg more of fuel compared to the time- and fuel-optimal trajectory of standard \ac{PDG}.
However, computing the solution with Bonmin took more than 10 minutes.
By using the S-B-MIQP algorithm, we can solve the MINLP in less than half the time required by Bonmin.
We converge to a different local optimum with a slightly higher objective.
Finally, the MPVC formulation can be solved very fast compared to the MINLP one, mostly because it only requires the solution of \acp{NLP}, but, Alg.~\ref{alg:homotopy} converges to a worse local minimum compared to the two \ac{MINLP} solutions.
Indeed, the MPVC trajectory traverses fewer divert-feasible regions, spends more fuel and exploits the slacks on the final position and velocity as the S-B-MIQP solution.
We also notice that the MPVC solution is more influenced by the initial guess compared to the MINLP one.
It may be possible to improve the MPVC solution by a different tuning of the weights in the cost function.
Fine tuning a specific formulation is out of the scope of this work, our main goal is to show that both formulations are solvable and each one has its own advantages and disadvantages.
From these simulations we see that the \ac{MPVC} formulation is faster to solve but requires more tuning to achieve solutions with objective values comparable to the \ac{MINLP} formulation.
Figure~\ref{fig: pyramid minlp traj} shows the position trajectories of the spacecraft for the \ac{PDG} problem, the MINLP formulation solved with S-B-MIQP and the MPVC formulation.
Figure~\ref{fig: pyramid minlp constraint} compares the constraint satisfaction for the three formulations, and in the bottom plot of the MINLP and MPVC formulation we report the activation of the indicator variables for the three different divert-feasible regions.
The thrust profile of the \ac{PDG} problem is bang-bang, which is usual for time optimal problem, while for MINLP and MPVC it is smooth.
Also, the glide slope constraint is not imposed at the final time step because it might interfere with the slacked terminal state constraints, leading to an infeasible problem.
From Figures~\ref{fig: pyramid minlp traj} and~\ref{fig: pyramid minlp constraint} we omitted the MINLP trajectories computed by Bonmin because they are similar to the one of S-B-MIQP.

\begin{table}
    \centering
    \small
    \begin{tabular}{lrrrr}\toprule
                                & MINLP - Bonmin    & MINLP - S-B-MIQP  & MPVC - Alg.~\ref{alg:homotopy}    & SOCP - IPOPT\\ \midrule
        Objective               & -15290.67         & -14565.96         & -7812.72                          &    $\shortminus$     \\
        Final position (m)      & (2.88, -5, 0)     & (5, -5, 0)        & (-5, 5, 0)                        & (0, 0, 0)    \\
        Final velocity (m/s)    & $10^{-3} \cdot (-7, 7, 7)$    & $10^{-3} \cdot (-7, 7, 7)$      & $  10^{-3} \cdot (7, -7, 7)$    & (0, 0, 0) \\
        Final mass (kg)         & 1561.52    & 1561.79      & 1560.79       & 1564.85       \\
        $\sum_i \sum_k \delta_{i,k}$ & 27    & 28           &  25           & $\shortminus$ \\ \midrule
        Runtime (s)             & 631.13    & 263.02        & 7.33          &  0.021        \\  \bottomrule
    \end{tabular}
    \caption{Results for divert-feasible landing with simplified regions, the last column corresponds to the solution of the standard \ac{PDG} problem.}
    \label{tab: results for simplified regions}
\end{table}

\begin{figure}[t!]
    \centering
    \subfigure[SOCP trajectory]{\includegraphics[width=0.9\textwidth]{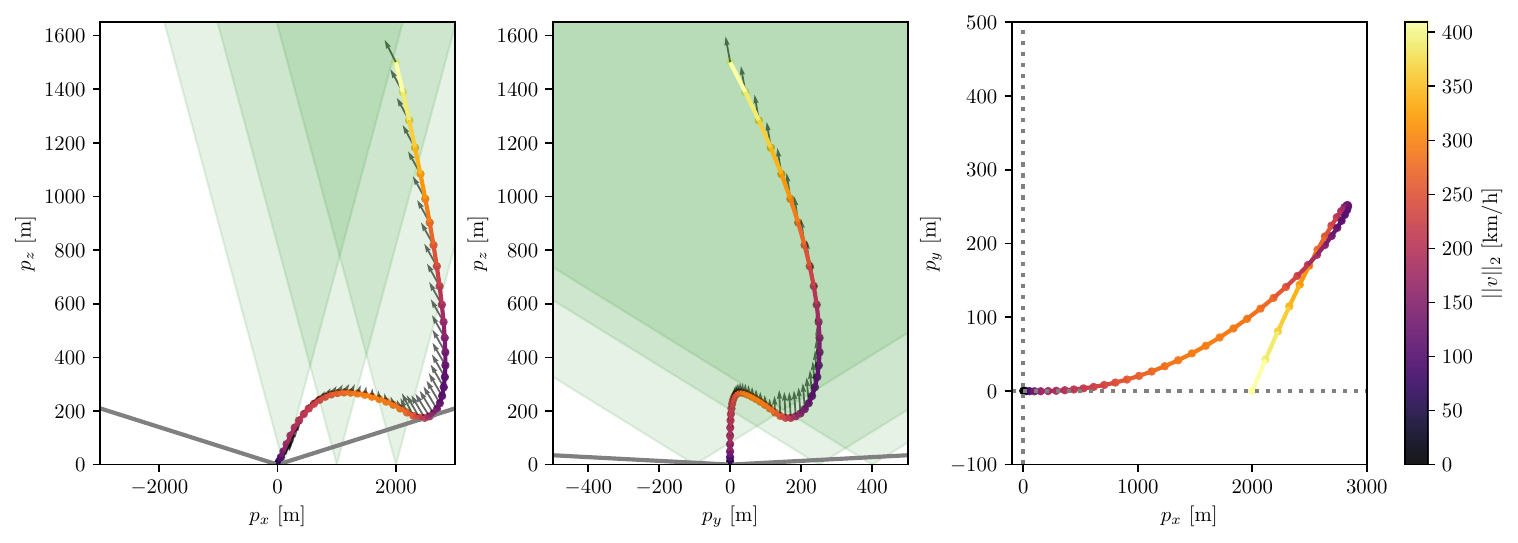}}
    \subfigure[MINLP trajectory]{\includegraphics[width=0.9\textwidth]{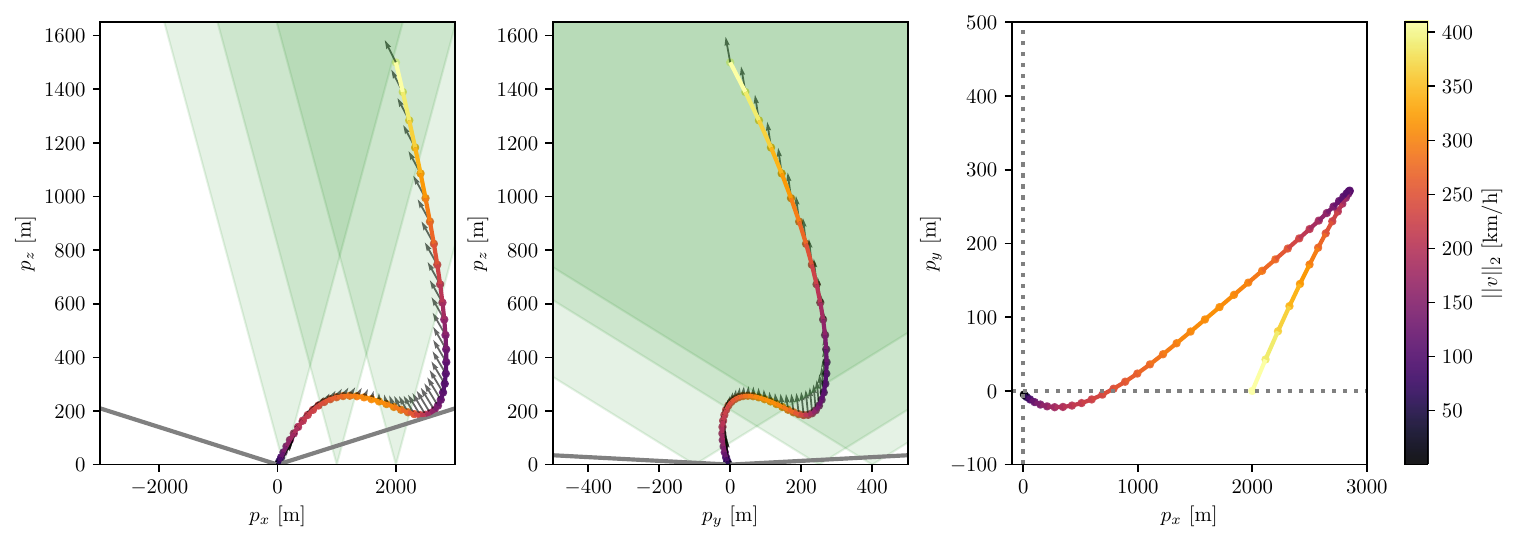}}
    \subfigure[MPVC trajectory]{\includegraphics[width=0.9\textwidth]{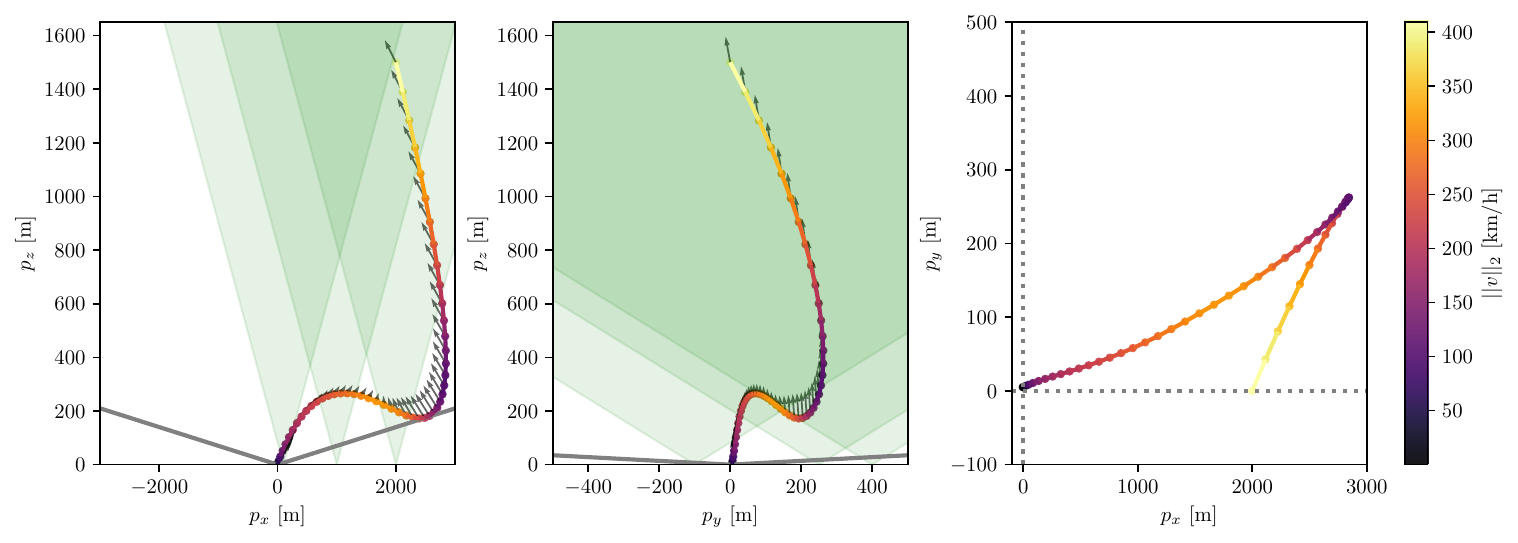}}
    \caption{Spacecraft trajectories for SOCP (top), MINLP~\eqref{op: minlp lander} (middle), and MPVC~\eqref{op: mpvc lander} (bottom).
    For each, the three plots show the trajectories of the position projected onto the $xz$, $yz$, and $xy$ plane, respectively. The color used to depict the trajectory shows the Euclidean norm of the lander's velocity. The direction of the grey arrows represents the pointing angle of the lander in every discretization point, and the length of the arrows represents the Euclidean norm of the thrust.
    The green areas represent the divert-feasible regions.
    The gray solid lines at the bottom of the left and middle plot represent the glide-slope constraint.
    The position is constrained to be above such lines.
    \label{fig: pyramid minlp traj}}
\end{figure}

\begin{figure}[t]
    \centering
    \subfigure[Constraint satisfaction for SOCP]{\includegraphics[width=0.33\textwidth]{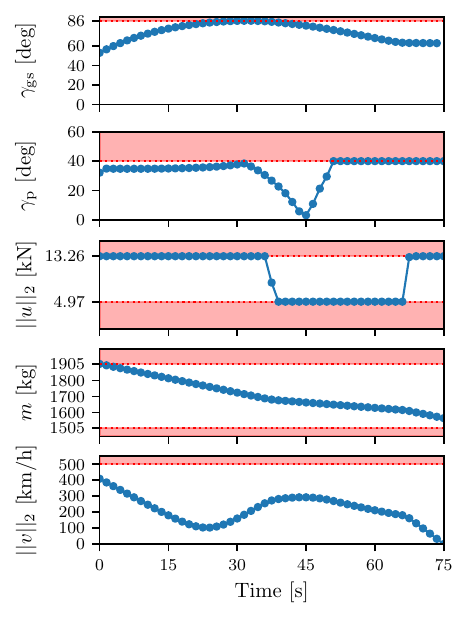}}
    \hfill
    \subfigure[Constraint satisfaction for MINLP~\eqref{op: minlp lander}.]{\includegraphics[width=0.33\textwidth]{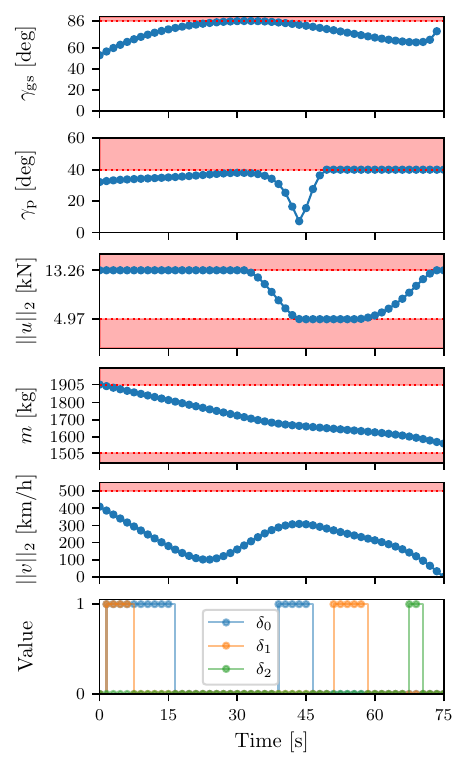}}
    \hfill
    \subfigure[Constraint satisfaction for MPVC~\eqref{op: mpvc lander}.]{\includegraphics[width=0.33\textwidth]{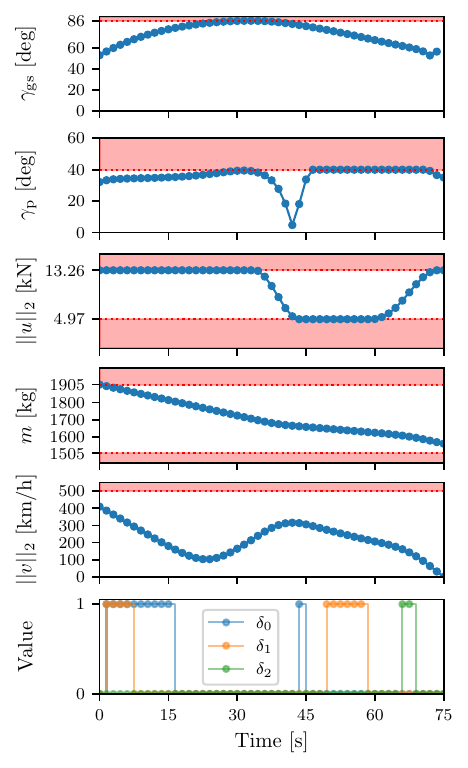}}
    \caption{Comparison of constraint satisfaction for the time- and fuel-optimal problem (no polyhedral constraints) SOCP (left), against MINLP~\eqref{op: minlp lander} (middle), and MPVC~\eqref{op: mpvc lander} (right).
    From top to bottom, the subplots represent the evolution of glide-slope angle, the evolution of the pointing angle, the Euclidean norm of the thrust, the mass depletion, the Euclidean norm of the lander's velocity. The bottom subplot in the middle and right plot represents the trajectory of the indicator variables $\boldsymbol{\delta}$, i.e., the membership in divert-feasible regions.
    \label{fig: pyramid minlp constraint}
    }
\end{figure}

\subsection{Simulations with realistic regions}\label{subsec: simulation with realistic regions}
Now, we want to show that the two approaches can also deal with the realistic divert-feasible regions.
These regions are computed based on reachable set analysis and they are convex polytopes in 6 dimensions, 3D position and 3D velocity.
The polytopes are computed according to the method described in~\cite{Lishkova2024} which yields polytopes in vertex representation.
Here, we transform the vertex representation to the halfspace representation using the \texttt{QHull} library from Python SciPy, with the options \texttt{QJ Qx C-0.00001 C0.001}.
Option \texttt{QJ} is used to increase numerical stability, option \texttt{Qx} allows merging of coplanar facets and it is controlled by the following two options.
\texttt{C-}, ``pre-merging'', allows for merging coplanar facets during the creation of the hull when the centrum of a facet is closer than 0.00001 to the centrum of a neighboring one.
\texttt{C}, ``post-merging'', is similar but applies the merging operation after the hull is constructed.
In this way, we avoid an explosion in the number of halfspaces required to describe the polytopes by slightly approximating the vertex representation.
Since in this work we are interested in showing the computational aspects of the problem, we did not compute tailored divert-feasible regions for the problem at hand.
Instead, we constructed our simulation with a polytope taken from~\cite{Lishkova2024}.
We made three copies and translated them in the position space in order to obtain three scattered regions around the primary landing target located in the origin.
The following three translation vectors expressed in meters $(-100, 0, -10)$, $(100, 100, -10)$, $(250, 0, -10)$ are applied in the position space of the polytope represented in Figure~\ref{fig: polytope projections}, in order to obtain three distinct polytopes.
\begin{table}[t]
    \centering
    \begin{tabular}{lll}
        \toprule
        V-repr. & H-repr. (original) & H-repr. (pruned) \\
        (80, 6) & (3586, 7) & (1672, 7) \\ \bottomrule
    \end{tabular}
    \caption{Dimensions of the realistic divert-feasible regions in different polytopic representations.}
    \label{tab: dimension realistic polytopes}
\end{table}
\begin{figure}[t]
    \centering
    \subfigure[Projection onto 3D position space.]{\includegraphics[trim={1.5cm 0cm 0cm 1cm}, clip, width=0.35\textwidth]{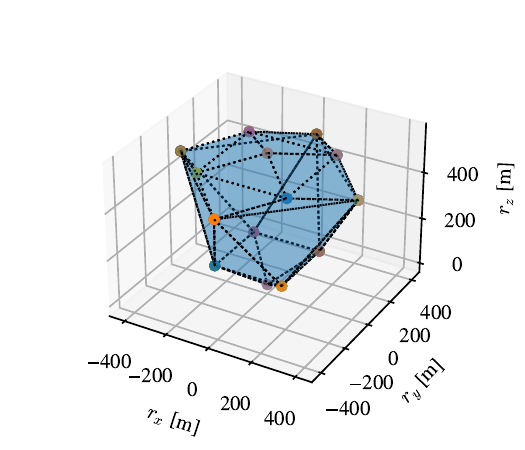}}
    \subfigure[Projection onto XZ-position and Z-velocity space.]{\includegraphics[trim={1.2cm 0cm 0 1cm}, clip, width=0.35\textwidth]{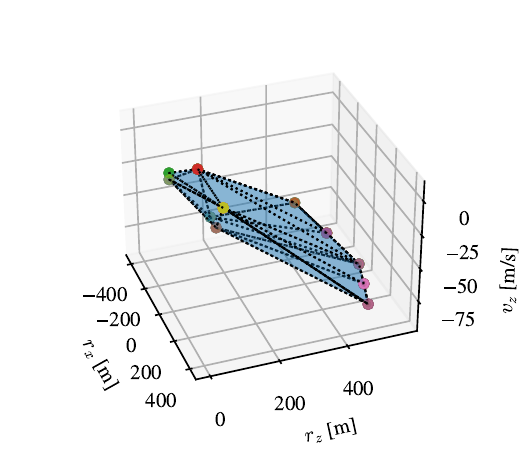}}
    \caption{Projections of the realistic divert-feasible region.~\label{fig: polytope projections}}
\end{figure}
\begin{table}[t!]
    \centering
    \small
    \begin{tabular}{lrr}\toprule
                    &MINLP - S-B-MIQP & MPVC with Alg.~\ref{alg:homotopy} \\ \midrule
        Objective               & -6248.36          & -4022.08  \\
        Final position (m)      & (1.93, -0.24, 0)  & (5, 4.88, 0) \\
        Final velocity (m/s)    & $10^{-3} \cdot (-7, 7, -7)$   & $  10^{-3} \cdot (-7, -7, -7)$ \\
        Final mass (kg)         & 1531.40                       & 1533.28     \\ 
        $\sum_i \sum_k \delta_{i,k}$   & 8      &  11.12     \\ \midrule
        Runtime (s)             & 3622.79           & 19845.02 \\ \bottomrule
    \end{tabular}
    \caption{Results for divert-feasible landing with realistic regions.}\label{tab: results for realistic regions}
\end{table}
Table~\ref{tab: dimension realistic polytopes} contains the dimension of the selected polytopes, Figure~\ref{fig: polytope projections} shows a projection of the polytope onto the 3D position space and onto the $xz$-position and $z$-velocity space.
The latter projection results in a very narrow polytope, and it helps to understand why computing trajectories that belong to these polytopes for as long as possible is not an easy task.

In the simulation with simplified regions we used 75 seconds corresponding to the optimal time found by solving the standard \ac{PDG} problem.
To obtain divert-feasible trajectories considering realistic divert-feasible regions, we need to extend the flight duration.
Hence, we set the final time $t_\mathrm{f} = 110$ s, which is a time budget 50\% higher compared to the trajectory computed by the standard \ac{PDG}.
Also, we increased the available fuel mass by 50 kg, thus $m_\mathrm{wet} = 1955$~kg.
By extending the flight time and increasing the fuel mass, we are implicitly giving more freedom to the optimizer to explore divert-feasible trajectories.
Finally, to keep a reasonable runtime we shorten the horizon to $N=30$.

As mentioned in Remark~\ref{remark: modelling choices}, we do not want to deal with the vertex representation because it requires auxiliary variables to construct the convex combination.
Also, checking the membership of a point into a polytope in vertex representation requires an equality constraint, which is generally hard to satisfy in the context of \ac{MINLP}.
The total number of auxiliary variables for the considered scenario, cf. Table~\ref{tab: dimension realistic polytopes}, would correspond to $n_\mathrm{p} n_\mathrm{c} N = 7200$ where $n_\mathrm{p} =3$ is the number of regions, each one described by $n_\mathrm{c}=80$ vertices, and $N=30$ is the length of the \ac{OCP} horizon.
Conversely, when using the halfspace representation, the number of variables is unchanged, but we end up having $1672 \cdot n_\mathrm{p} N = 150480$ linear inequalities.
When using a MINLP solver based on a decomposition scheme like S-B-MIQP, the internal \ac{MIP} solver usually has powerful presolve routines that can dramatically reduce the number of these inequalities, since many of these are never active.
For the \ac{NLP} solver where presolve routines are not present, the main issue is the memory footprint and the time spent into linear algebra routines.
If the problem fits in the available memory, usually the linear inequalities are not challenging in the step computation of an advanced solver like \texttt{IPOPT}.

Again, we solved both~\eqref{op: minlp lander} and~\eqref{op: mpvc lander}, the results are collected in Table~\ref{tab: results for realistic regions}.
In this case Bonmin could not solve the problem, therefore we omitted it from Table~\ref{tab: results for realistic regions}.
The MINLP formulation is solved by S-B-MIQP within \texttt{CAMINO} which returns a better solution compared to Alg.~\ref{alg:homotopy} applied to the MPVC formulation.
In Alg.~\ref{alg:homotopy}, we could not push $\tau$ low enough to obtain purely binary values for the indicator variables because of numerical problems.
Therefore, we stopped the homotopy loop as soon as we achieved a $\tau < \tau_\mathrm{min} = 0.1$.
We argue that for this instance where we have many constraints describing each divert-feasible region, the MPVC formulation is less efficient because it creates too many nonconvexities which need to be handled by the solver.
So, even if the constraints are relaxed in the homotopy loop, it is hard to find a descent direction that reduces $\tau$ in every constraint.
On the contrary, for the MINLP formulation the constraints describing the regions are linear inequalities.
Specifically, for S-B-MIQP linear inequalities do not create an issue either for the master problems, which are \acp{MIP}, or for the auxiliary \ac{NLP}.

\vspace{-0.2cm}
\section{Conclusion, discussion and outlook}\label{sec: conclusions}
In this work we presented formulations for modeling decision-making problems with performance objectives and constraints expressed through logical expressions.
The approach is particularly relevant in various scenarios, as demonstrated through aerospace case studies.
We further showed how the decision-making formulation can be adapted for solution via Newton-based optimization methods, introducing both \ac{MINLP} and \ac{MPVC} formulations.
Numerical methods for solving these problems were discussed, and the formulations were analyzed and validated through  a \acf{PDG} case study with divert-feasible regions for Mars landing.
The numerical experience showed that the \ac{MPVC} approach is generally faster for computing locally optimal solutions for problems with limited amount of nonconvexities, i.e., of logical implications.
However, the performance of the \ac{MPVC} formulation decreases dramatically when the number of nonconvexities is high, and it might be necessary to stop the homotopy procedure at fractional solutions, as shown in Sec.~\ref{subsec: simulation with realistic regions}.
The \ac{MINLP} formulation solved with the S-B-MIQP algorithm showed to be more reliable than the homotopy method for \ac{MPVC} even in case of numerous nonconvexities.
Also, the \ac{MINLP} formulation is more intuitive as decisions are directly represented by binary variables rather than by structured nonconvexities.
We highlight that the performance of the MINLP approach might change according to the \ac{MIP} solver adopted in the S-B-MIQP algorithm.

\vspace{-0.2cm}
\section*{Acknowledgments}
Andrea Ghezzi has received funding from the European Union's Horizon 2020 research and innovation programme under the Marie Sklodowska-Curie grant agreement ELO-X No. 953348.
Andrea Ghezzi thanks the financial support of MERL during the internship period when this work was developed.
Armin Nurkanovi\'c, Moritz Diehl acknowledge fundings from DFG via Research Unit FOR 2401, project 424107692, 504452366 (SPP 2364), and 525018088, from BMWK via 03EI4057A and 03EN3054B, and from the EU via ELO-X 953348.

\bibliography{lander}

\end{document}